\documentclass[12pt]{amsart}
\pdfoutput=1 
\usepackage{hyperref}
\usepackage{amsmath,amsfonts,amssymb,verbatim}
\usepackage{mathrsfs}
\usepackage[english]{babel}

\newtheorem{theorem}{Theorem}[section]
 \newtheorem{corollary}[theorem]{Corollary}
 \newtheorem{lemma}[theorem]{Lemma}
 \newtheorem{proposition}[theorem]{Proposition}
 \theoremstyle{definition}
 \newtheorem{definition}[theorem]{Definition}
 \theoremstyle{remark}
 \newtheorem{remark}[theorem]{Remark}
 
 \numberwithin{equation}{section}

\def \bC {\mathbb C}

\def \bH {\mathbb H}

\def \bN {\mathbb N}

\def \bR {\mathbb R}

\def \bZ {\mathbb Z}

\def \cA {\mathcal A}
\def \cB {\mathcal B}

\def \cD {\mathcal D}

\def \cH {\mathcal H}
\def \cI {\mathcal I}

\def \cL {\mathcal L}

\def \cQ {\mathcal Q}
\def \cR {\mathcal R}
\def \cS {\mathcal S}
\def \cR {\mathcal R}
\def \cR {\mathcal R}

\def \cV {\mathcal V}

\def \cX {\mathcal X}

\def \fg {\mathfrak g}
\def \fh {\mathfrak h}

\def \fS {\mathfrak S}

\def \fU {\mathfrak U}

\def \id {\text{\rm I}}
\def \dom {\text{\rm Dom}}

\def \range {\text{\rm Range}}

\def\G{{G}}

\def\sB{{\mathscr B}}

\def \sL{\mathscr L}

\def\Rn{{{\mathbb R}^n}}

\def\Re{{{\rm Re}\,}}

\begin{document}

\title[]
{Sobolev spaces on graded groups}

\author[]
{V\'eronique Fischer and Michael Ruzhansky}

\address{Department of Mathematics,
Imperial College London,
180 Queen's Gate, 
London SW7 2AZ, 
United Kingdom}

\subjclass[2010]{Primary: 43A80; Secondary:  22E30, 35H05, 44A25}                         

\keywords{Harmonic analysis on nilpotent Lie groups, Sobolev spaces, graded groups,
Rockland operators, heat semigroup}

\begin{abstract}
We study the $L^p$-properties  of positive Rockland operators
and define Sobolev spaces on general graded groups.
This generalises the case of 
sub-Laplacians on stratified groups 
studied by G. Folland in \cite{folland_75}.
We show that the defined Sobolev spaces are actually independent
of the choice of a positive Rockland operator. Furthermore, we show
that they are interpolation spaces and establish duality and Sobolev
embedding theorems in this context.
\end{abstract}

\maketitle

\section{Introduction}

One can define Sobolev spaces on $\bR^n$ in various equivalent ways
(see e.g. \cite{mazya}),
for example using the Euclidean Fourier transform on $\bR^n$
or the properties of the Laplace operator.
This can also be done on compact Lie groups 
(see e.g. \cite{rt:book, rt:groups}, also for the corresponding global
theory of pseudo-differential operators).
Replacing the Laplace operator with a (left-invariant) sub-Laplacian 
on a stratified nilpotent Lie group
and using the associated heat semigroup,
Folland showed in  \cite{folland_75} 
that the corresponding spaces 
are different from  their Euclidean (abelian) counterpart
 but share many properties with them.
See also \cite{saka}.
Using Littlewood-Payley decompositions,
this was generalised in \cite{furioli+melzi+veneruso}
to the context of Lie groups of polynomial growth,
which in general do not have a (global) homogeneous structure.

Our purpose here is to define functional  spaces of Sobolev type 
on graded (homogeneous) Lie groups.
This class of groups has proved to be a natural setting 
to generalise several questions of the Euclidean harmonic analysis
and contains many interesting examples:
the abelian Euclidean case of $\bR^n$, the Heisenberg group with its natural structure as a stratified group, and more generally any stratified group, but also for example the Heisenberg group with a different graded non-stratified structure (see, e.g., $\tilde \bH_1$ in Section \ref{subsec_comparison}). 
These groups appear naturally in the geometry of certain symmetric domains and in some sub-elliptic partial differential equations.
In fact,
one can argue easily that the analysis on stratified groups such as in  \cite{folland_75} 
was born out of studying operators based on sums of squares of vector fields 
on certain manifolds, nowadays called Heisenberg manifolds.
Similarly, our present  analysis should help understanding the case of more general operators,
of higher degrees as differential operators and in terms of homogeneity. 

In this paper, we define functional spaces on graded groups and show some important properties such as interpolation, duality, 
adapted Sobolev embeddings and so on.
Hence this justifies our choice of calling these functional spaces Sobolev spaces.
Our construction uses positive Rockland operators
instead of sub-Laplacians 
as the latter are no longer always homogeneous.
Although our analysis is closely related to Folland's  \cite{folland_75},
one important obstacle is the fact that positive Rockland operators may be of high degree, 
and not only 2 as in the case of sub-Laplacians.
This has deep implications. For example, a Rockland operator may not have a unique homogeneous fundamental solution.
Also, the powerful Hunt theorem
\cite{hunt}
 is no longer available for operators of order larger than 2.

Naturally, when we consider a graded group which is stratified, 
we recover the Sobolev spaces defined by Folland in  \cite{folland_75} 
which then coincide with the Sobolev spaces obtained in \cite{furioli+melzi+veneruso}
on any Lie group of polynomial growth.
However, for a general graded (non-stratified) group, 
our Sobolev spaces may differ from the ones in \cite{furioli+melzi+veneruso}, 
see Section \ref{subsec_comparison} in this paper.
They will also be slightly different  from the Goodman-Sobolev spaces defined 
by Goodman 
on graded Lie groups for integer exponents only
in \cite[Sec. III. 5.4]{goodman_LNM},
see again Section \ref{subsec_comparison}.   

The advantage 
of the Sobolev spaces defined in this paper is a collection of
natural functional analytic properties making them useful in applications:
for example, Goodman's versions
\cite{goodman_LNM} are not interpolation spaces while our spaces are. 
Moreover our Sobolev are adapted to the homogeneous structure
(we could in fact also define homogeneous Sobolev spaces, see our last remark in Section \ref{subsec_comparison}).
For instance, homogeneous left-invariant differential operators maps 
Sobolev spaces to other Sobolev spaces with `a loss of derivatives' corresponding to the homogeneous degree. 
This is not the the case in general with  \cite{furioli+melzi+veneruso}.

For the sake of completeness,
our analysis  includes the definition and some properties for the case $p=\infty$.
This is new already for the stratified case, although
this could have been done by adapting Folland's methods.

\medskip

This paper is organised as follows.
After some preliminaries about graded groups and their homogeneous structure in Section \ref{sec_preliminary},
we first define the fractional powers of a positive Rockland operator in Section \ref{SEC:powers-Rockland}, as well as its Riesz and Bessel potentials.
This enables us to define our Sobolev spaces in Section \ref{SEC:Sobolev},
where we also show that they satisfy the properties expected from Sobolev spaces,
e.g. interpolation and duality amongst others.
In the last section, we show a `Sobolev' embedding theorem and we also compare our Sobolev spaces with other known spaces.

\section{Preliminaries}
\label{sec_preliminary}

In this section, after defining graded Lie groups, 
we recall their homogeneous structure as well as the definition and some properties of
their Rockland operators.

\subsection{Graded and homogeneous groups}

Here we recall briefly the definition of graded nilpotent Lie groups
and their natural homogeneous structure.
A complete description of the notions of graded and homogeneous nilpotent Lie groups may be found in \cite[ch1]{folland+stein_bk82}.

We will be concerned with graded Lie groups $G$
which means that $G$ is a connected and simply connected 
Lie group 
whose Lie algebra $\mathfrak g$ 
admits an $\bN$-gradation
$\mathfrak g= \oplus_{\ell=1}^\infty \mathfrak g_{\ell}$
where the $\mathfrak g_{\ell}$, $\ell=1,2,\ldots$, 
are vector subspaces of $\mathfrak g$,
all but finitely many equal to $\{0\}$,
and satisfying 
$[\mathfrak g_{\ell},\mathfrak g_{\ell'}]\subset\mathfrak g_{\ell+\ell'}$
for any $\ell,\ell'\in \bN$.
This implies that the group $G$ is nilpotent.
Examples of such groups are the Heisenberg group and, more generally,
all stratified groups (which by definition correspond to the case $\fg_1$ generating the full Lie algebra $\fg$).

We construct a basis $X_1,\ldots, X_n$  of $\fg$ adapted to the gradation,
by choosing a basis $\{X_1,\ldots X_{n_1}\}$ of $\mathfrak g_1$ (this basis is possibly reduced to $\{0\}$), then 
$\{X_{n_1+1},\ldots,  X_{n_1+n_2}\}$ a basis of $\mathfrak g_2$
(possibly $\{0\}$ as well as the others)
and so on.
Via the exponential mapping $\exp_G : \mathfrak g \to G$, we   identify 
the points $(x_{1},\ldots,x_n)\in \bR^n$ 
 with the points  $x=\exp_G(x_{1}X_1+\cdots+x_n X_n)$ in $G$.
Consequently we allow ourselves to denote by $C(G)$, $\cD(G)$ and $\cS(G)$ etc,
the spaces of continuous functions, of smooth and compactly supported functions or 
of Schwartz functions on $G$ identified with $\bR^n$,
and similarly for distributions with the duality notation 
$\langle \cdot,\cdot\rangle$.

This basis also leads to a corresponding Lebesgue measure on $\mathfrak g$ and the Haar measure $dx$ on the group $G$,
hence $L^p(G)\cong L^p(\bR^n)$.
The group convolution of two functions $f$ and $g$, 
for instance   integrable, 
is defined via 
$$
 (f*g)(x):=\int_\G f(y) g(y^{-1}x) dy.
$$
The convolution is not commutative: in general, $f*g\not=g*f$.
However, apart from the lack of commutativity,  
group convolution and the usual convolution on $\Rn$
share many properties. For example, we have
\begin{equation}
\langle f * g, h \rangle
=\langle f ,h*\tilde g\rangle,
\quad\mbox{with}\quad \tilde g(x)=g(x^{-1})
\label{eq_int_f*gh_f}.
\end{equation}
And the Young convolutions inequalities hold:
if $f_1\in L^p(G)$ and $f_2\in L^q(G)$
with $1\leq p,q,r\leq \infty$ 
and $\frac 1p+\frac 1q=\frac 1r+1$,
 then $f_1*f_2\in L^r(G)$ and
\begin{equation}
\label{eq_Young_ineq}
\|f_1*f_2\|_r\leq \|f_1\|_p \|f_2\|_q.
\end{equation}

The coordinate function $x=(x_1,\ldots,x_n)\in G\mapsto x_j \in \bR$
is denoted by $x_j$.
More generally we define for every multi-index $\alpha\in \bN_0^n$,
$x^\alpha:=x_1^{\alpha_1} x_2 ^{\alpha_2}\ldots x_{n}^{\alpha_n}$, 
as a function on $G$.
Similarly we set
$X^{\alpha}=X_1^{\alpha_1}X_2^{\alpha_2}\cdots
X_{n}^{\alpha_n}$ in the universal enveloping Lie algebra $\fU(\fg)$ of $\mathfrak g$.

For any $r>0$, 
we define the  linear mapping $D_r:\mathfrak g\to \mathfrak g$ by
$D_r X=r^\ell X$ for every $X\in \mathfrak g_\ell$, $\ell\in \bN$.
Then  the Lie algebra $\mathfrak g$ is endowed 
with the family of dilations  $\{D_r, r>0\}$
and becomes a homogeneous Lie algebra in the sense of 
\cite{folland+stein_bk82}.
We re-write the set of integers $\ell\in \bN$ such that $\fg_\ell\not=\{0\}$
into the increasing sequence of positive integers
 $\upsilon_1,\ldots,\upsilon_n$ counted with multiplicity,
 the multiplicity of $\fg_\ell$ being its dimension.
 In this way, the integers $\upsilon_1,\ldots, \upsilon_n$ become 
 the weights of the dilations and we have $D_r X_j =r^{\upsilon_j} X_j$, $j=1,\ldots, n$,
 on the chosen basis of $\fg$.
 The associated group dilations are defined by
$$
D_r(x)=
r\cdot x
:=(r^{\upsilon_1} x_{1},r^{\upsilon_2}x_{2},\ldots,r^{\upsilon_n}x_{n}),
\quad x=(x_{1},\ldots,x_n)\in G, \ r>0.
$$
In a canonical way  this leads to the notions of homogeneity for functions and operators.
For instance
the degree of homogeneity of $x^\alpha$ and $X^\alpha$,
viewed respectively as a function and a differential operator on $G$, is 
$[\alpha]=\sum_j \upsilon_j\alpha_{j}$.
Indeed, let us recall 
that a vector of $\mathfrak g$ defines a left-invariant vector field on $G$ 
and, more generally, 
that the universal enveloping Lie algebra of $\mathfrak g$ 
is isomorphic with the left-invariant differential operators; 
we keep the same notation for the vectors and the corresponding operators. 

Recall that a \emph{homogeneous pseudo-norm} on $G$ is a continuous function $|\cdot| : G\rightarrow [0,+\infty)$ homogeneous of degree 1
on $G$ which vanishes only at 0. This often replaces the Euclidean pseudo-norm in the analysis on homogeneous Lie groups:
\begin{proposition}
\label{prop_homogeneous_pseudo_norm}
\begin{enumerate}
\item 
Any homogeneous pseudo-norm $|\cdot|$ on $G$ satisfies a triangle inequality up to a constant:
$$
\exists C\geq 1 \quad \forall x,y\in G\quad
|xy|\leq C (|x|+|y|).
$$
It partially satisfies the reverse triangle inequality:
\begin{equation}
\label{eq_reverse_triangle}
\forall b\in (0,1) \quad
\exists C=C_b\geq 1 \quad \forall x,y\in G
\quad |y|\leq b |x| \Longrightarrow
\big| |xy| - |x|\big|\leq C|y|
.
\end{equation}
\item
Any two homogeneous pseudo-norms $|\cdot|_1$ and $|\cdot|_2$ are equivalent in the sense that
$$
\exists C>0 \quad \forall x\in G\quad
C^{-1} |x|_2\leq |x|_1\leq C |x|_2.
$$
\item 
 A concrete example of a homogeneous pseudo-norm is given via
$$
|x|_{\nu_o}:=\Big(\sum_{j=1}^n x_j^{2\nu_o/\upsilon_j}\Big)^{1/{2\nu_o}},
$$
with $\nu_o$ a common multiple to the weights $\upsilon_1,\ldots,\upsilon_n$.
\end{enumerate}
\end{proposition}

 Various aspects of analysis on $G$ can be developed in a comparable way with the Euclidean setting, see \cite{coifman+weiss-LNM71}, sometimes replacing the topological dimension 
$$
n:=\dim G =\sum_{\ell=1}^\infty\dim \fg_\ell
$$
of the group $G$ by its \emph{homogeneous dimension}
$$
Q:=\sum_{\ell=1}^\infty \ell \dim \fg_\ell
=  \upsilon_1 +\upsilon_2 +\ldots +\upsilon_n 
 .
$$

For example, there is an analogue of polar coordinates on homogeneous groups with $Q$ replacing $n$:

\begin{proposition}
\label{prop_polar_coord}
Let $|\cdot|$ be a fixed homogeneous pseudo-norm on $G$.
Then there is a (unique) positive Borel measure $\sigma$ on 
the unit sphere 
$\fS:=\{x\in G\, : \, |x|=1\}$,
such that for all $f\in L^1(G)$, we have
\begin{equation}
\label{formula_polar_coord}
\int_G f(x)dx
=
\int_0^\infty \int_\fS f(ry) r^{Q-1} d\sigma(y) dr
.
\end{equation}
\end{proposition}

Another example is the following  property regarding kernels or operators  of type $\nu$
(see  \cite{folland_75} and \cite[Chapter 6 A]{folland+stein_bk82}):
\begin{definition}
\label{def_kernel_type_nu}
A distribution $\kappa\in \cD'(G)$ 
which is smooth away from the origin and  homogeneous of degree $\nu-Q$
is called a \emph{kernel of type} $\nu\in \bC$ on $G$.
The corresponding convolution operator
$f\in \cD(G)\mapsto f*\kappa$ is called an \emph{operator of type} $\nu$.
\end{definition}

\begin{theorem}
\label{thm_op_type_Lp_bdd}
An operator of type $\nu$ with $\nu\in [0,Q)$
is $(-\nu)$-homogeneous  
and
extends to a bounded operator from $L^p(G)$ to $L^q(G)$
whenever $p,q\in (1,\infty)$ satisfy
$\frac 1p -\frac 1q = \frac {\Re \nu}Q$.
\end{theorem} 

Exactly as in the Euclidean setting,
in the case  $\Re \nu\in (0,Q)$, 
any smooth function away from the origin which is $(\nu-Q)$-homogeneous
defines a distribution.
However in the case $\Re\nu=0$,
one needs to add a condition to guarantee the same property:
\begin{proposition}
\label{prop_-Qhomo_distrib}
Let $\kappa$ be a smooth function away from the origin homogeneous of degree $\nu$ with $\Re \nu=-Q$.
It coincides with the restriction to $G\backslash\{0\}$
of a distribution in $\cD'(G)$
if and only if its mean value is zero, 
that is, 
when 
$\int_{\fS} \kappa \, d\sigma=0$ where $\sigma$ is the measure on the unit sphere $\fS$ of a homogeneous pseudo-norm given 
by the polar change of coordinates, see Proposition \ref{prop_polar_coord}.
This condition is independent of the choice of a homogeneous pseudo-norm.
\end{proposition} 

The problems about (group) convolving distributions 
on $G$ are essentially the same as 
in the case of the abelian convolution on $\bR^n$.
The convolution $\tau_1*\tau_2$ of two distributions $\tau_1,\tau_2\in \cD'(G)$ is well defined as a distribution provided that at most one of them has compact support.
However, additional assumptions must be imposed in order to define convolutions of distributions with non-compact supports.
Furthermore, the associativity of the group convolution product 
  law 
  \begin{equation}
  \label{eq_associative_convolution}
(\tau_1*\tau_2)*\tau_3=\tau_1*(\tau_2*\tau_3)
,
\end{equation}
holds when at most one of the $\tau_j$'s has non-compact support
but not necessarily when only one of the $\tau_j$'s has compact support 
even if each convolution in \eqref{eq_associative_convolution} 
could have a meaning.
 
 The following proposition establishes that there is no such pathology appearing when considering convolution with kernel of type $\nu$ with $\Re \nu\in [0,Q)$. This will be useful in the sequel.
\begin{proposition}
\label{prop_conv_kernel_typnu}
Let $G$ be a homogeneous group.
\begin{itemize}
\item[(i)]
Suppose $\nu\in \bC$
with $0\leq \Re \nu<Q$, $p\geq 1$, $q>1$,
and $r\geq 1$ given by $\frac 1r=\frac 1p+\frac 1q-\frac{\Re \nu}Q-1$.
If $\kappa$ is a kernel of type $\nu$, $f\in L^p(G)$,
and $g\in L^q(G)$, then $f*(g*\kappa)$ and $(f*g)*\kappa$ are well defined as elements of $L^r(G)$, and they are equal.
\item[(ii)]
Suppose $\kappa_{1}$ is a kernel of type ${\nu_1}\in \bC$
with $\Re {\nu_1}>0$ and $\kappa_{2}$ is a kernel of type ${\nu_2}\in \bC$ 
with $\Re{\nu_2}\geq 0$.
We assume $\Re({\nu_1}+{\nu_2})<Q$.
Then $\kappa_{1}*\kappa_{2}$ is well defined as a kernel of type 
${\nu_1}+{\nu_2}$.
Moreover if $f\in L^p(G)$ where $1<p<Q/(\Re({\nu_1}+{\nu_2}))$
then $(f*\kappa_{1})*\kappa_{2}$ and $f*(\kappa_{1}*\kappa_{2})$
belong to $ L^q(G)$, $\frac 1q=\frac 1p -\frac{\Re({\nu_1}+{\nu_2})}Q$,
 and they are equal.
\end{itemize}
\end{proposition}

The approximations of the identity may be constructed on $G$ as on their Euclidean counterpart,
replacing the topological dimension and the abelian convolution 
with the homogeneous dimension and the group convolution:

\begin{lemma}
\label{lem_approximatation}
Let $\phi\in L^1(\G)$.
Then the functions $\phi_t$, $t>0$, defined via
$\phi_t(x)=t^{-Q} \phi(t^{-1} x )$,
are integrable and $\int \phi_t=\int \phi$ is independent of $t$.
Furthermore, 
for any $f$ in $L^p(G)$, $C_o(G)$, $\cS(G)$ or $\cS'(G)$,
the sequence of functions
$f*\phi_t$ and $\phi_t*f$, $t>0$,  converges towards $(\int \phi)\, f$ 
as $t \to 0$ in 
$L^p(G)$, $C_o(G)$, $\cS(G)$ and $\cS'(G)$ respectively.
\end{lemma}

In Lemma \ref{lem_approximatation} and in the whole  paper,  
$C_o(G)$ denotes the space of continuous functions on $G$ which vanish at infinity. 
This means that $f\in C_o(G)$ when for every $\epsilon>0$ there exists a compact set $K$ outside which 
we have $|f|<\epsilon$.
Endowed with the supremum norm $\|\cdot\|_\infty=\|\cdot\|_{L^\infty(G)}$, 
it is a Banach space.

Recall that $\cD(G)$, the space of smooth and compactly supported functions, 
is dense in $L^p(G)$ for $p\in [1,\infty)$ and in $C_o(G)$
(in which case we set $p=\infty$).

In Theorem \ref{thm_heat_kernel},
we will see that the heat semi-group associated to a positive Rockland operator gives an approximation of the identity which is commutative.

\subsection{Rockland operators}

Here we recall the definition of Rockland operators
and their main properties.

The definition of a Rockland operator uses the representations of the group.
Here we consider only continuous unitary representations of $G$.
We will often denote by $\pi$ such a representation, 
by $\cH_\pi$ its Hilbert space and by $\cH_\pi^\infty$
 the subspace of smooth vectors.
The corresponding infinitesimal representation on the Lie algebra $\fg$
and its extension to the universal enveloping Lie algebra $\fU(\fg)$
are also denoted by $\pi$.
We recall that $\fg$ and $\fU(\fg)$ are identified with the spaces of left-invariant vector fields 
and of left-invariant differential operators on $G$ respectively.

\begin{definition}
A \emph{Rockland operator}
\index{Rockland!operator}
 $\cR$ on $G$ is 
a left-invariant differential operator $T$
which is homogeneous of positive degree and satisfies the Rockland condition:
\begin{center}
(R) 
for each unitary irreducible representation $\pi$ on $G$,
except for the trivial representation, \\
the operator $\pi(T)$ is injective on $\cH_\pi^\infty$, 
that is,
\end{center}
$$
 \forall v\in \cH_\pi^\infty\qquad
\pi(T)v = 0 \ \Longrightarrow \ v=0
.
$$
\end{definition}

Although the definition of a Rockland operator would make sense on a homogeneous Lie group (in the sense of  \cite{folland+stein_bk82}), 
it turns out (see \cite{miller}, see also \cite[Lemma 2.2]{TElst+Robinson})  
that the existence of a (differential) Rockland operator on a homogeneous group 
implies that the homogeneous group may be assumed to be graded.
This explains why we have chosen to restrict our presentation to graded Lie groups.

Some authors may have different conventions than ours regarding Rockland operators:
for instance some choose to consider right-invariant operators
and some definitions of a Rockland operator involves only the principal part
of the operator.
The analysis however would be exactly the same.
In a different direction, Glowacki studied non-differentiable ($L^2$-bounded) operators which satisfy the Rockland condition in \cite{glowacki_89,glowacki_91}.

In 1977, 
Rockland conjectured in \cite{rockland} that the property in (R)
which nowadays bears his name 
is equivalent to the hypoellipticity of the operator.
This was eventually proved by Helffer and Nourrigat in 
\cite{helffer+nourrigat-79}.
Hence
Rockland operators may be viewed as 
an analogue of elliptic operators (with a high degree of homogeneity)
in a non-abelian subelliptic context.
In the stratified case, one can check easily that
any (left-invariant negative)  \emph{sub-Laplacian}, 
that is
\begin{equation}
\label{eq_def_lih_subLapl}
\cL=Z_1^2+\ldots+Z_{n'}^2
\quad\mbox{with} \ Z_1,\ldots,Z_{n'} \
\mbox{forming any basis of the first stratum} \ \fg_1,
\end{equation}
 is a Rockland operator.
More generally it is not difficult to see that 
 the operator
\begin{equation}
\label{eq_cR_ex}
\sum_{j=1}^n
(-1)^{\frac{\nu_o}{\upsilon_j}}
 c_j X_j^{2\frac{\nu_o}{\upsilon_j}}
\quad\mbox{with}\quad c_j>0
,
\end{equation}
is a Rockland operator of homogeneous degree $2\nu_o$
if $\nu_o$ is any common multiple of $\upsilon_1,\ldots, \upsilon_n$.
Hence Rockland operators do exist on any graded Lie group 
(not necessarily stratified).
Furthermore, if $\cR$ is a Rockland operator, 
then one can show easily that its powers $\cR^k$, $k\in \bN$, and 
its complex conjugate $\bar\cR$ are also Rockland operators.

If  a Rockland operator $\cR$ which is formally self-adjoint,
that is, $\cR^*=\cR$
as elements of the universal enveloping algebra $\fU(\fg)$, 
is fixed,
then it admits a self-adjoint extension
on $L^2(G)$
\cite[p.131]{folland+stein_bk82}.
In this case we will denote by $\cR_2$ the self-adjoint extension and 
  by $E$ 
its spectral measure:
\begin{equation}
\label{eq_spectral_meas_E}
\cR_2 = \int_\bR \lambda dE (\lambda).
\end{equation}

\subsection{Positive Rockland operators and their heat kernels}
\label{subsec_pos_R_op}

In this section we summarise properties of positive Rockland operators that
are important for our analysis.

Recall that an operator $T$ on a Hilbert space $\cH$
is positive
when for any vectors $v,v_1,v_2\in \cH$ in the domain of $T$, we have
$
(Tv_1, v_2)_\cH=(v_1,Tv_2)_\cH
\;\mbox{and}\;
(Tv, v)_\cH
\geq 0
 .
$
 If $T$ is a left-invariant differential operator acting on $G$, 
 then $T$ is positive when $T$ is  formally self-adjoint, that is, $T^*=T$ in $\fU(\fg)$, and satisfies
$$
\forall f\in\cD(G) \qquad
\int_G Tf(x) \overline{f(x)}\, dx
\geq 0
 .
 $$
Note that if $G$ is stratified and $\cL$ is a (left-invariant negative) sub-Laplacian, 
then $-\cL$ is a positive Rockland operator.
The example in \eqref{eq_cR_ex} 
is also a positive Rockland operator.
Hence positive Rockland operators always exist on any graded Lie group.
Moreover if $\cR$ is a positive Rockland operator, 
then its powers $\cR^k$, $k\in \bN$, and 
its complex conjugate $\bar\cR$ are also positive Rockland operators.

Let us fix a positive Rockland operator  $\cR$ on $G$.
By  functional calculus 
(see \eqref{eq_spectral_meas_E}),
we can define the spectral multipliers 
$$
e^{-t\cR_2}:=\int_0^\infty e^{-t\lambda} dE(\lambda), \quad t>0 ,
$$
which form the \emph{heat semigroup} of $\cR$.
The operators $e^{-t\cR_2}$ are invariant under left-translations and are bounded on $L^2(G)$. 
Therefore the Schwartz kernel theorem implies that each operator 
$e^{-t\cR_2}$ admits
a unique distribution $h_t\in \cS'(G)$
as its convolution kernel:
$$
e^{-t\cR_2}f = f*h_t, \quad t>0, \ f\in \cS(G).
$$ 
The distributions $h_t$, $t>0$, are called the \emph{heat kernels} of $\cR$.
We summarise their main properties in the following theorem:
\begin{theorem}
\label{thm_heat_kernel}
  Let $\cR$ be a positive Rockland operator on $G$  which is homogeneous of degree $\nu\in \bN$.
Then each distribution $h_t$ is Schwartz and we have:
\begin{eqnarray}
\label{eq_heat_kernel_hths_ht+s}
\forall s,t>0
\qquad h_t*h_s&=&h_{t+s}
,
\\
\label{eq_heat_kernel_homogeneity}
\forall x\in G,\, t,r>0
\qquad
h_{r^\nu t}(r x)&=& r^{-Q} h_t(x)
,
\\
\label{eq_homogeneity_sym+pos}
\forall x\in G
\qquad
h_t(x)&=&\overline{h_t(x^{-1}) }
,
\\
\label{eq_heat_kernel_int}
\int_G h_t(x) dx&=&1
.
\end{eqnarray}
The function $h: G\times \bR\to\bC$ defined by 
$$
h(x,t):=\left\{\begin{array}{ll}
h_t(x) 
& \mbox{if} \ t>0 \ \mbox{and} \ x\in G ,\\
0 
& \mbox{if} \ t\leq 0 \ \mbox{and} \ x\in G , \\
\end{array}\right.
$$
is smooth on $(G\times \bR) \backslash\{(0,0)\}$ and satisfies 
$(\cR+\partial_t)h=\delta_{0,0}$ where $\delta_{0,0}$ is the delta-distribution at $(0,0)\in G\times \bR$.
Having fixed a homogeneous pseudo-norm $|\cdot|$ on $G$,
we have for any $N\in \bN_0$, $\alpha\in \bN_{0}^n$ and $\ell\in \bN_0$:
 \begin{equation}
 \label{eq_control_heat_kernel}
 \exists C=C_{\alpha,N,\ell}>0 \quad
 \forall t\in (0,1]\quad
 \sup_{|x|=1} |\partial_t^\ell X^\alpha h_t(x)|\leq C_{\alpha,N} t^N
  .
\end{equation}
Consequently
\begin{equation}
\label{eq_heat_kernel_homogeneity-1}
\forall x\in G,\; t>0 \qquad h_t(x)=t^{-\frac Q\nu} h_{1}(t^{-\frac 1\nu} x),
\end{equation}
and for $x\in G\backslash\{0\}$ fixed,
\begin{equation}
\label{eq_control_heat_kernel_xfixed0}
X_x^\alpha h(x,t)=
\left\{\begin{array}{l}
O(t^{-\frac{Q+[\alpha]} \nu}) \ \mbox{as}\ t\to \infty ,\\
O(t^N) \ \mbox{for all N}\in \bN_0\ \mbox{as}\ t\to 0 .
\end{array}\right.
\end{equation}
Inequalities \eqref{eq_control_heat_kernel_xfixed0} are also valid 
for any $x$ in a fixed compact subset of $G\backslash\{0\}$.
\end{theorem}

\begin{remark}
\label{rem_proof_thm_heat_kernel_schwartz}
If the group is stratified and $\cR
=-\cL$ where $\cL$ is a sub-Laplacian, 
then $\cR$ is of order two and the proof relies on Hunt's theorem 
\cite{hunt}, cf. \cite[ch1.G]{Folland-Stein:BK-Hardy-on-homo-groups}.
In this case, the heat kernel is real-valued and moreover non-negative.
The heat semigroup is then a semigroup of contraction which preserves positivity.
\end{remark}

Theorem \ref{thm_heat_kernel} was proved by Folland and Stein 
in \cite[ch4.B]{folland+stein_bk82}.
In their proof, 
they also show the following technical property
which we will also use later on:
\begin{lemma}
\label{lem_weak_sobolev_embedding}
Let $\cR$ be a positive Rockland operator of a graded Lie group $G\sim \bR^n$
with homogeneous degree $\nu$.
If $m$ is a positive integer such that $m\nu \geq \lceil \frac n2 \rceil$, then
the functions in  the domain of $\cR^m$ are continuous on $\Omega$
and  for any compact subset $\Omega$  of $G$,
there exists a constant $C=C_{\Omega, \cR,G,m}$ such that
$$
\forall \phi\in \dom (\cR^m)\quad
\quad
\sup_{x\in \Omega}|\phi(y)|
\leq C \left(\|\phi\|_{L^2}+\|\cR^m\phi\|_{L^2}\right).
$$
\end{lemma}
This is a weak form of Sobolev embeddings. 
We will later on obtain stronger results of this kind
in Theorem \ref{thm_Sobolev-ineq}.

\medskip

We end this section with the following result of Liouville's type:

\begin{theorem}
\label{thm_liouville_homogeneous_for_R}
If $\cR$ is a positive Rockland operator and 
$f\in \cS'(G)$ a distribution satisfying  $\cR f=0$ then $f$ is a polynomial.
\end{theorem}

\begin{proof}
As $\cR$ is a positive Rockland operator, $\bar \cR=\cR^t$ is also Rockland 
and they are both hypoelliptic, see \cite{helffer+nourrigat-79}.
The conclusion follows 
by applying the Liouville theorem for homogeneous Lie groups
proved by Geller in \cite{geller_83}.
\end{proof}

\section{Fractional powers of positive Rockland operators}
\label{SEC:powers-Rockland}
\index{fractional powers of operators}

In this section we aim at defining fractional powers of positive Rockland operators.
We will carry out the construction on the scale of $L^{p}$-spaces for 
$1\leq p\leq\infty$, with $L^{\infty}(G)$ substituted by the space $C_{o}(G)$ of
continuous functions vanishing at infinity. 
Then we discuss the essential
properties of such an extension. 
Eventually we define its complex powers, and the
corresponding Riesz and Bessel potentials.

\subsection{Positive Rockland operators on $L^p$}
\label{subsec_op_cRp}

Here we define and study  the analogue $\cR_p$ of the operator $\cR$ on $L^p(G)$
or $C_o(G)$.
This analogue  will be defined as  the infinitesimal generator
of the heat convolution semigroup.
Hence 
we start by  proving the following properties:

\begin{proposition}
\label{prop_heat_kernel_semigroup_Lp}
The operators $f\mapsto f*h_t$, $t>0$, form a 
 strongly continuous semi-group on $L^p(G)$ for any $p\in [1,\infty)$
 and on $C_o(G)$ if $p=\infty$.
 This semi-group is also equibounded:
$$
\forall t>0, \ \forall f\in L^p(G) \ \mbox{or} \ C_o(G)\qquad
\|f*h_t\|_p\leq \|h_1\|_1 \|f\|_p.
$$
 Furthermore for any $p\in [1,\infty]$ 
 (finite or infinite) and any $f\in \cD(G)$,
   \begin{equation}
\label{eq_cv_R_p}
 \lim_{t\to 0}\Big\|\frac 1 t(f*h_t-f)-\cR f\Big\|_p = 0.
\end{equation}
\end{proposition}

\begin{proof}[Proof of Proposition \ref{prop_heat_kernel_semigroup_Lp}]
If $f\in \cD(G)$, then $f\in \dom(\cR)\subset \dom (\cR_2)$
and for any $s,t>0$, by functional calculus,
$$
f*h_{t+s}=e^{-(t+s)\cR_2} f = e^{-t\cR_2}e^{-s\cR_2} f
=
(f*h_s)*h_t
$$
and, by the Young convolution inequalities for $p\in [1,\infty]$ (see \eqref{eq_Young_ineq}), 
$$
\|f*h_t\|_p \leq \|h_t\|_1 \|f\|_p
$$
with $\|h_t\|_1=\|h_1\|_1<\infty$ by 
 Theorem \ref{thm_heat_kernel}.
 By density of $\cD(G)$ in $L^p(G)$ for $p\in [1,\infty)$ and $C_o(G)$ for $p=\infty$, 
this  implies that the operators $f\mapsto f*h_t$, $t>0$, form a 
 strongly continuous equibounded semi-group on $L^p(G)$ for any $p\in [1,\infty)$
 and on $C_o(G)$.
 
Let us prove the convergence in \eqref{eq_cv_R_p} for $p=\infty$.
Let $f\in \cD(G)$. By Lemma \ref{lem_weak_sobolev_embedding},
for any compact subset $\Omega\subset G$, 
\begin{eqnarray*}
&&\sup_{\Omega}
\left|\frac 1t \left(f*h_t -f\right) -\cR f\right|
\\&&\leq 
C \left(
\left\|\frac 1t \left(f*h_t -f\right) -\cR f\right\|_2
+
\left\|\frac 1t \cR^m\left(f*h_t -f\right) -\cR^{m+1} f\right\|_2
\right),
\end{eqnarray*}
where $m$ is an integer such that $m\nu \geq \lceil \frac n2\rceil$.
Since $\cD(G)\subset \dom (\cR)$ and $e^{-t\cR_2}f=f*h_t$, 
we have for any integer $m'\in \bN_0$:
\begin{eqnarray*}
&&\frac 1t \cR^{m'}\left(f*h_t -f\right) -\cR^{m'+1} f
=
\frac 1t \cR_2^{m'}\left(e^{-t\cR_2} f -f\right) -\cR_2^{m'+1} f
\\
&&\quad=
\frac 1t \left(e^{-t\cR_2}\cR_2^{m'} f -\cR_2^{m'}f\right) -\cR_2^{m'+1} f
=
\frac 1t \left((\cR^{m'} f)*h_t -\cR^{m'} f\right) -\cR^{m'+1} f.
\end{eqnarray*}
This last expression converges to zero in $L^2(G)$ as $t\to 0$.
Therefore 
$$
\sup_{\Omega}
\left|\frac 1t \left(f*h_t -f\right) -\cR f\right|\longrightarrow_{t\to 0} 0 .
$$

We  fix a homogeneous pseudo-norm $|\cdot|$ on $G$, for example the one in Part (3) of Proposition \ref{prop_homogeneous_pseudo_norm}.
We denote by  $\bar B_R:=\{x\in G, \ |x|\leq R\}$  the closed ball 
about 0 of radius $R$.
We now fix $R\geq 1$ such that $\bar B_R$ contains the support of $f$.
Let $C_o=C_b$ be the constant in the reverse triangle inequality,
see \eqref{eq_reverse_triangle},
for $b=\frac 12$.
We choose $\Omega=\bar B_{2C_oR}$ the closed ball about 0 and with radius $2C_oR$.
If $x\not\in \Omega$, then since $f$ is supported in $\bar B_R\subset\Omega$,
$$
\left(\frac 1t \left(f*h_t -f\right) -\cR f\right)(x)=
\frac 1t f*h_t(x)
=
\frac 1t \int_{|y|\leq R} f(y) h_t(y^{-1}x) dy,
$$
hence
$$
\left|\frac 1t f*h_t(x)\right|
\leq
\frac{\|f\|_\infty }t
 \int_{|y|\leq R} |h_t(y^{-1}x)| dy
=
\frac{\|f\|_\infty }t
 \int_{|xt^{\frac 1\nu} z^{-1}|\leq R} |h_1(z)| dz,
$$
as $h_t$ satisfies \eqref{eq_heat_kernel_homogeneity-1}.
The reverse triangle inequality, see
\eqref{eq_reverse_triangle},
implies $
\{|xt^{\frac 1\nu} z^{-1}|\leq R\}
\subset
\{|z|>t^{-\frac 1\nu} R/2\}$.
 Since $h_1$ is Schwartz, we must have 
 $$
\exists C\quad 
\quad
| h_1(z)|\leq C |z|^{-\alpha},
 $$
 for $\alpha=-Q-2\nu$ for instance.
 This together with the polar change of variable
 (see Proposition \ref{prop_polar_coord})
 yield  
 $$
 \int_{|z|>t^{-\frac 1\nu} R/2} |h_1(z)| dz
 \leq 
 C \int_{r=t^{-\frac 1\nu} R/2}^\infty r^{-\alpha -Q-1} dr
 =C' t^2.
$$
Consequently taking the supremum in the complementary on $\Omega$
$$
\sup_{\Omega^c}
\left|\frac 1t \left(f*h_t -f\right) -\cR f\right|
\leq C't 
\longrightarrow_{t\to 0} 0 .
$$
This shows 
the convergence in \eqref{eq_cv_R_p} for $p=\infty$.

We now proceed in a similar way to prove the convergence in \eqref{eq_cv_R_p} for $p$ 
finite.
As above we fix $f\in \cD(G)$ supported in $\bar B_R$.
We decompose
$$
\|\frac 1t \left(f*h_t -f\right) -\cR f\|_p
\leq
\|\frac 1t(f*h_t-f)-\cR f\|_{L^p(\bar B_{2 C_{o}R})}
+
\|\frac 1t(f*h_t-f)-\cR f\|_{L^p(B^c_{2 C_{o}R})}
.
$$
For the first term, 
$$
\|\frac 1t(f*h_t-f)-\cR f\|_{L^p(\bar B_{2 C_{o}R})}
\leq
|\bar B_{2 C_{o}R}|^{\frac 1p}
\|\frac 1t(f*h_t-f)-\cR f\|_\infty
\underset{t\to0}\longrightarrow 0 ,
$$
as we have already proved the convergence in \eqref{eq_cv_R_p} for $p=\infty$.
For the second term, 
we obtain for the reasons explained in the case $p=\infty$:
\begin{eqnarray*}
&&\|\frac 1t(f*h_t-f)-\cR f\|_{L^p(B^c_{2 C_{o}R})}
=
\frac 1t \|f*h_t \|_{L^p(B^c_{2C_{o}R})}
\\
&&\qquad=\frac 1t 
\left(\int_{|x|>2 C_{o}R}
\left|\int_{|y|<R} f(y) \  h_t(y^{-1}x) dy\right|^p dx
\right)^{\frac 1p}\\
&&\qquad\leq 
C_{0,2} \frac {\|f\|_\infty}t
\left(\int_{|x|>2 C_{o}R}
 \int_{|z|>t^{-\frac 1\nu} R/2} |h_1(z)| dz
\right)^{\frac 1p}
\leq 
C_{0,2} \frac {\|f\|_\infty}t
\left(C' t^p\right)^{\frac 1p},
\end{eqnarray*}
choosing this time $\alpha=-Q-p\nu$.
This yields the convergence in \eqref{eq_cv_R_p} for $p$ 
finite.
\end{proof}

\begin{definition}
\label{def_cRp}
Let $\cR$ be a positive Rockland operator on $G$.

For $p\in [1,\infty)$,
we denote by $\cR_p$ the operator such that $-\cR_p$
is the infinitesimal  generator of 
the semi-group of operators $f\mapsto f*h_t$, $t>0$, on 
the Banach space $L^p(G)$.

We also denote by $\cR_{\infty_o}$ the operator such that $-\cR_{\infty_o}$
is the infinitesimal  generator of 
the semi-group of operators $f\mapsto f*h_t$, $t>0$, on the Banach space $C_o(G)$.
\end{definition}

For the moment it seems that $\cR_2$ denotes 
the self-adjoint extension of $\cR$ on $L^2(G)$ and 
minus the generator of  $f\mapsto f*h_t$, $t>0$, on $L^2(G)$.
In the sequel,
in fact in Theorem~\ref{thm_cRp} below, 
we show that the two operators coincide  and there is no conflict of notation.

\begin{theorem}
\label{thm_cRp}
  Let $\cR$ be a positive Rockland operator on $G$
  and   $p\in[1,\infty)\cup\{\infty_o\}$.
\begin{itemize}
\item[(i)] 
The operator $\cR_p$ is closed.
The domain of $\cR_p$ contains $\cD(G)$, and for $f\in \cD(G)$ we have
$\cR_p f=\cR f$.

\item[(ii)] 
The operator $\bar \cR_p$ is positive and Rockland.
Moreover $-\bar \cR_p$
is the infinitesimal generator of the strongly continuous semi-group
$\{f\mapsto f*\bar h_t\}_{t>0}$
on $L^p(G)$ for $p\in [1,\infty)$ and on $C_o(G)$ for $p=\infty_o$.

\item[(iii)] 
If $p\in (1,\infty)$ then the dual of $\cR_p$ is $ \bar \cR_{p'}$.
The dual of $\cR_{\infty_o}$ restricted to $L^1(G)$ is $\bar \cR_1$.
The dual of $\cR_1$ restricted to $C_o(G)\subset L^\infty(G)$ is $\bar \cR_{\infty_o}$.

\item[(iv)] 
If $p\in [1,\infty)$,
the operator $\cR_p$ is the maximal restriction of $\cR$ to $L^p(G)$,
that is, the domain of $\cR_p$ consists of all the functions $f\in L^p(G)$ 
such that the distributional derivative $\cR f$ is in $L^p(G)$ and $\cR_p f=\cR f$.

The operator $\cR_{\infty_o}$ is the maximal restriction of $\cR$ to $C_o(G)$,
that is, the domain of $\cR_{\infty_o}$ consists of all the function $f\in C_o(G)$ 
such that the distributional derivative $\cR f$ is in $C_o(G)$ and $\cR_p f=\cR f$.

\item[(v)] 
If $p\in [1,\infty)$,
the operator $\cR_p$ is the smallest closed extension of $\cR|_{\cD(G)}$ on $L^p(G)$.
For $p=2$,  $\cR_2$ is the self-adjoint extension of $\cR$ on $L^2(G)$.
\end{itemize}
\end{theorem}

\begin{proof}
Part (i) is a consequence of Proposition \ref{prop_heat_kernel_semigroup_Lp}.
Intertwining with the complex conjugate,
this implies that
$\{f\mapsto f*\bar h_t\}_{t>0}$ is also a strongly continuous semi-group on $L^p(G)$
whose infinitesimal operator coincides with $-\bar \cR=-\cR^t$ on $\cD(G)$.
This shows Part (ii).

For Part (iii),
we observe that 
using \eqref{eq_int_f*gh_f} and \eqref{eq_homogeneity_sym+pos},
we have
\begin{equation}
\label{eq_f1*ht,f_2}
\forall f_1,f_2\in \cD(G)\qquad
\langle f_1*h_t ,f_2\rangle
=
\langle f_1 ,f_2* \bar h_t\rangle
.
\end{equation}
Thus we have for any $f,g\in \cD(G)$ and $p\in [1,\infty)\cup \{\infty_o\}$
$$
\langle \frac 1t(e^{-t\cR_p}f -f) , g\rangle
=
\frac 1t\langle f*h_t -f , g\rangle
=
\frac 1t\langle f, g*\bar  h_t-g  \rangle
=
\frac 1t\langle f, e^{-t\bar \cR_{p'}}g-g  \rangle .
$$
Here  the brackets refer to the duality in the sense of distribution. 
 Taking the limit as $t\to 0$ of the first and last expressions proves Part (iii).

We now prove Part (iv) for any $p\in [1,\infty)\cup\{\infty_o\}$.
Let  $f\in \dom (\cR_p)$ 
and $\phi\in \cD(G)$.
Since $\cR$ is formally self-adjoint, 
we know that $\cR ^t= \bar \cR$,
and by Part (i), we have 
$ \cR_q \phi =\cR \phi$
for any $q\in [1,\infty)\cup\{\infty_o\}$. Thus by Part (iii)
we have 
$$
\langle \cR_p f,  \phi\rangle 
= 
\langle  f,  \bar \cR_{p'}  \phi\rangle 
=
\langle  f, \cR^t  \phi\rangle 
=
\langle \cR f,  \phi \rangle 
,
$$
and
$\cR_p f=\cR f$ in the sense of distributions.
Thus 
$$
\dom (\cR_p)\subset \{f\in L^p(G)\ : \ \cR f\in L^p(G)\}
.
$$

We now prove the reverse inclusion. 
Let $f\in L^p(G)$ such that $\cR f\in L^p(G)$.
Let also $\phi\in \cD(G)$. 
The following computations are justified by  
the properties of $\cR$ and $h_t$  (see Theorem~\ref{thm_heat_kernel}),
Fubini's Theorem, and  \eqref{eq_f1*ht,f_2}:
\begin{eqnarray*}
&&\langle f*h_t-f, \phi \rangle
=
\langle f, \phi* \bar h_t -\phi \rangle
=
\langle f, \int_0^t \partial_s (\phi* \bar h_s)  ds \rangle=
\langle f, \int_0^t -\bar \cR (\phi*  \bar h_s)  ds \rangle\\
&&\qquad
=
-\langle f,  \bar\cR \int_0^t  (\phi*  \bar h_s)  ds \rangle=
-\langle \cR f, \int_0^t  \phi*  \bar h_s  ds \rangle
=
 -\int_0^t \langle \cR f,   \phi*  \bar h_s   \rangle ds
\\
&&\qquad=
 -\int_0^t \langle (\cR f) * h_s,   \phi  \rangle ds
 =
  -\langle \int_0^t (\cR f) * h_s  ds,   \phi  \rangle
  .
 \end{eqnarray*}
Therefore,
$$
\frac 1t (f*h_t-f)=-\frac 1t \int_0^t (\cR f) * h_s  ds.
$$
This converges towards 
$-\cR f$ in $L^p(G)$ as
$t\to0$ by the general properties of averages 
of strongly continuous semigroups on a Banach space.
This shows $f\in \dom(\cR_p)$ and concludes the proof of (iv).

Part (v) follows from (iv).
This also shows that the self-adjoint extension of $\cR$ coincides with
$\cR_2$ as defined in Definition \ref{def_cRp} and concludes the proof of Theorem~\ref{thm_cRp}.
\end{proof}

Theorem \ref{thm_cRp} has the following consequences
which will enable us to define the fractional powers of $\cR_p$.

\begin{corollary}
\label{cor_cRp_injective}
We keep the same setting and notation as in Theorem~\ref{thm_cRp}.
\begin{itemize}
\item[(i)] 
The operator $\cR_p$ is injective on $L^p(G)$ for $p\in [1,\infty)$
 and $\cR_{\infty_o}$ is injective on $C_o(G)$, namely,
$$
\mbox{for}\ p\in [1,\infty)\cup\{\infty_o\}\ :\qquad
\forall f\in \dom (\cR_p) \qquad
\cR_p f=0\Longrightarrow f=0
.
$$
\item[(ii)] 
If $p\in (1,\infty)$ then
the operator  $\cR_p$ has dense range in $L^p(G)$.
The operator $\cR_{\infty_o}$ has dense range in $C_o(G)$.
The closure of the range of $\cR_1$ is 
the closed subspace $\{\phi\in L^1(G) \, : \, \int_G \phi=0\}$ of $L^1(G)$.
\item[(iii)] 
For $p\in [1,\infty)\cup\{\infty_o\}$,
and any $\mu>0$, 
the operator $\mu\id+\cR_p$ is invertible on $L^p(G)$, $p\in [1,\infty)$, 
and on $C_o(G)$ for $p=\infty_o$,
and the operator norm of $(\mu \id+\cR_p)^{-1}$
is 
\begin{equation}
\label{eq_Komatsu_noneg}
\|(\mu \id+\cR_p)^{-1}\|_{\sL(L^p(G))} \leq \|h_1\| \mu^{-1}
\quad\mbox{or}\quad
\|(\mu \id+\cR_{\infty_o})^{-1}\|_{\sL(C_o(G))} \leq \|h_1\| \mu^{-1}.
\end{equation}
\end{itemize}
\end{corollary}

\begin{proof}
Let $ f\in \dom (\cR_p)$ be
such that $\cR_p f=0$
for $p\in [1,\infty)\cup\{\infty_o\}$.
By Theorem~\ref{thm_cRp} (iv), $f\in \cS'(G)$ and $\cR f=0$.
Consequently by Liouville's theorem, see Theorem \ref{thm_liouville_homogeneous_for_R}, 
 $f$ is a polynomial. 
Since $f$ is also in $L^p(G)$ for $p\in [1,\infty)$
or in $C_o(G)$ for $p=\infty_o$, $f$ must be identically zero.
This proves (i).

For (ii), let $\Psi$ be a bounded linear functional on $L^p(G)$ if $p\in [1,\infty)$
or on $C_o(G)$ if $p=\infty_o$ such that $\Psi$ vanishes identically on 
$\range(\cR_p)$.
Then $\Psi$ can be realised as the integration against a function $f\in L^{p'}(G)$ if $p\in [1,\infty)$ or 
a measure also denoted by $f\in M(G)$ if $p=\infty_o$.
Using the distributional notation, we have
$$
\Psi(\phi)=\langle f,\phi\rangle
\qquad \forall \phi\in L^p(G)\quad \mbox{or}\quad \forall\phi\in C_o(G)
.
$$
Then for any $\phi\in \cD(G)$, we know that $\phi\in \dom (\cR_p)$
and $\cR_p \phi=\cR\phi$
by Theorem~\ref{thm_cRp} (i)
thus
$$
0=\Psi(\cR_p(\phi))=\langle f, \cR (\phi)\rangle 
= \langle \bar \cR f, \phi\rangle
,
$$
since $\cR^t=\bar \cR$.
Hence $\bar \cR f =0$.
By Liouville's theorem, see Theorem \ref{thm_liouville_homogeneous_for_R}, 
this time applied to the positive Rockland operator $\bar \cR$,
we see that $ f$ is a polynomial.
This implies that $f\equiv0$,
 since $f$ is also a function in $L^{p'}(G)$ in the case  $p\in (1,\infty)$,
 whereas
 for $p=\infty_o$, $f$ is  in $M(G)$ thus an integrable polynomial on $G$.
For $p=1$, $f$ being a measurable bounded function and a polynomial,
 $f$ must be constant, i.e. $f\equiv c$ for some $c\in \bC$.
This shows that if $p\in (1,\infty)\cup\{\infty_o\}$
then $\Psi=0$ and $\range(\cR_p)$ is dense in $L^p(G)$ or $C_o(G)$,
whereas if $p=1$ then $\Psi:L^1(G)\ni \phi\mapsto c\int_G \phi$.
This shows (ii) for $p\in (1,\infty)\cup\{\infty_o\}$.

Let us study more precisely the case $p=1$.
It is easy to see that 
$$
\int_G X \phi (x) dx = -\int_{G} \phi(x)\  (X1)(x) dx = 0
$$
holds for any $\phi\in L^1(G)$ such that $X\phi\in L^1(G)$.
Consequently, for any $\phi\in \dom (\cR_1)$,
we know that $\phi$ and $\cR \phi$ are in $L^1(G)$ thus
 $\int_G \cR_1 \phi=0$.
 So the range of $\cR_1$ is included in 
$$
S:=\{\phi\in L^1(G) \ : \ \int_G \phi=0\}
\supset \range(\cR_1).
$$
Moreover, if $\Psi_1$ a bounded linear functional on $S$
 such that
$\Psi_1$ is identically 0 on $\range(\cR_1)$,
by the Hahn-Banach Theorem,
it can be extended into a bounded linear function $\Psi$ on $L^1(G)$.
As $\Psi$  vanishes identically on $\range(\cR_1)\subset S$, 
we have already proven that $\Psi$ must be of the form 
 $\Psi:L^1(G)\ni \phi\mapsto c\int_G \phi$ for some constant $c\in \bC$
 and its restriction to $S$ is $\Psi_1\equiv 0$.
 This concludes the proof of Part (ii).

Let us prove Part (iii).
Integrating the formula
$$
(\mu+\lambda)^{-1} = \int_0^\infty e^{-t(\mu+\lambda)} dt
$$
against the spectral measure $dE(\lambda)$ of $\cR_2$, 
we have formally
\begin{equation}
\label{eq_mu+cR_inv}
(\mu\id  +\cR_2)^{-1}=\int_0^\infty  e^{-t (\mu \id + \cR_2)}dt,
\end{equation}
and the convolution kernel of the operator on the right-hand side is (still formally) given by
$$
\kappa_\mu(x):=\int_0^\infty e^{-t\mu} h_t(x) dt.
$$

From the properties of the heat kernel $h_t$
(see Theorem \ref{thm_heat_kernel}), we see that
$\kappa_\mu$ is continuous on $G$ and that 
$$
\|\kappa_\mu\|_1\leq 
\int_0^\infty e^{-t\mu} \|h_t\|_1 dt 
=\|h_1\| 
\int_0^\infty e^{-t\mu} dt 
=\frac{\|h_1\|}{\mu} <\infty
.
$$
Therefore $\kappa_\mu \in L^1(G)$
and
the operator
$\int_0^\infty  e^{-t (\mu \id + \cR_2)}dt$
is bounded on $L^2(G)$.
Furthermore  Formula \eqref{eq_mu+cR_inv}
holds (it suffices to consider integration over $[0,N]$ with $N\to\infty$).

For any $\phi\in \cD(G)$ and $p\in [1,\infty)\cup\{\infty_o\}$, 
Theorem \ref{thm_cRp} (iv) implies 
$$
(\mu\id+\cR_p)\phi=(\mu\id+\cR)\phi= (\mu\id+\cR_2)\phi
\in \cD(G),
$$
thus
$$
\kappa_\mu *\left( (\mu\id+\cR_p)\phi\right)
=
\kappa_\mu * \left((\mu\id+\cR_2)\phi\right)
=
\phi.
$$
Hence the inverse of the operator $\mu\id+\cR_p$ coincide 
with the convolution operator $\phi\mapsto \phi*\kappa_\mu$ 
which is bounded on $L^p(G)$ if $p\in [1,\infty)$ and on $C_o(G)$ if $p=\infty_o$.
Furthermore the operator norm of the latter is 
$\leq \|\kappa_\mu\|_1\leq 
\|h_1\| \mu^{-1}$.
\end{proof}

\subsection{Fractional powers of operators $\cR_{p}$}

In this section we study the fractional powers of the operators $\cR_p$ and $\id+\cR_p$.

\begin{theorem}
\label{thm_fractional_power}
  Let $\cR$ be a positive Rockland operator on a graded group $G$.
We consider the operators $\cR_p$ defined
in Definition \ref{def_cRp}.
Let $p\in [1,\infty)\cup\{\infty_o\}$.
\begin{enumerate}
\item
\label{item_thm_fractional_power_common_cR_id+cR}
Let $\cA_p$ denote either $\cR_p$ or $\id+\cR_p$.

\begin{itemize}

\item For every $a\in \bC$, the operator $\cA_p^a$ is closed
and  injective with $(\cA_p^{a})^{-1}=\cA_p^{-a}$.
We have $\cA_p^0=\id$, and for any $n\in \bN$, $\cA^n$ coincides with the usual powers of differential operators on $\cD(G)$.
Furthermore, the operator $\cA_p^a$ is invariant under left translations.

\item  
For any $a,b\in \bC$, in the sense of operator graph, we have
 $ \cA_p^a \cA_p^b\subset\cA_p^{a+b}.$
 If $p\not=1$ then  the closure of  $\cA_p^a \cA_p^b$ is $\cA_p^{a+b}$.
\item Let $a_o\in \bC_+$.

\begin{itemize}

\item
If $\phi\in \range (\cA^{a_o})$ then
$\phi\in \dom(\cA^a)$ for all $a\in\bC$ 
with $-\Re a_o<\Re a<0$ 
and the function
$a \mapsto \cA^a \phi$ is holomorphic in $\{a\in \bC:\, -\Re a_o < \Re a<0\}$.

\item
If $\phi\in \dom (\cA^{a_o})$ then
$\phi\in \dom(\cA^a)$ for all $a\in\bC$ 
with $0<\Re a<\Re a_o$ 
and the function
$a \mapsto \cA^a \phi$ is holomorphic in $\{a\in \bC:\, 
0<\Re a < \Re a_o\}$.

\end{itemize}

\item
If $p\in (1,\infty)$ then the dual of $\cA_p$ is $ \bar \cA_{p'}$.
The dual of $\cA_{\infty_o}$ restricted to $L^1(G)$ is $\bar \cA_1$.
The dual of $\cA_1$ restricted to $C_o(G)\subset L^\infty(G)$ is $\bar \cA_{\infty_o}$.

\item 
If $ a, b\in \bC_+$ with $\Re b>\Re a$, then
$$
\exists C=C_{a, b} >0
\quad\forall \phi\in \dom(\cA_p^ b)\quad
\|\cA_p^ a \phi\| \leq C \|\phi\|^{1-\frac{\Re  a}{\Re b}}
\|\cA_p^ b \phi \|^{\frac{\Re  a}{\Re b}}
.
$$
\end{itemize}

\item 
\label{item_thm_fractional_domain_id+L_L}
For each $a\in \bC_+$, the operators $(\id+\cR_p)^a$ and $\cR_p^a$ are unbounded and their domains satisfy 
$\dom\left[(\id+\cR_p)^a\right]
=\dom (\cR_p^a)
=\dom\left[(\cR_p+\epsilon\id)^a\right]$
for all $\epsilon>0$,
and all these domains contain 
$\cS(G)$.

\item 
\label{item_thm_fractional_power_id_L_Gamma}
If $0<\Re a<1$ and $\phi\in \range(\cR_p)$ then
$$
\cR_p^{-a} \phi =\frac 1{\Gamma(a)}
\int_0^\infty t^{a-1} e^{-t\cR_p}\phi\, dt
,
$$
in the sense that $\lim_{N\to\infty}\int_0^N$ converges in the norm of $L^p(G)$ or $C_o(G)$.
\item
\label{item_thm_fractional_power_id_I+L_Gamma}
If $a\in \bC_+$, then the operator $(\id+\cR_p)^{-a} $ is bounded and for any $\phi\in \cX$ with $\cX=L^p(G)$ or $ C_o(G)$, we have
$$
(\id+\cR_p)^{-a} \phi =\frac 1{\Gamma(a)}
\int_0^\infty t^{a-1} e^{-t(\id+\cR_p)} \phi\, dt,
$$
in the sense of absolute  convergence: 
$\int_0^\infty t^{a-1} 
\|e^{-t(\id+\cR_p)} \phi\|_{\cX} dt <\infty$.

\item 
\label{item_thm_fractional_power_cRitau}
For any $\tau\in \bR$ and $p\in (1,\infty)$,
the operator $\cR_p^{i\tau}$ is bounded on $L^p(G)$.
Moreover 
$$
\exists C,\theta>0\quad
\forall \tau\in \bR\qquad
\|\cR_p^{i\tau}\|_{\sL(L^p)} \leq C e^{\theta|\tau|}.
$$
and  for any $a\in \bC$, 
$$
\dom(\cR_p^a)=\dom (\cR_p^{\Re a}).
$$
\item 
\label{item_thm_fractional_power_(id+cR)itau}
For any $\tau\in \bR$ and $p\in (1,\infty)$,
the operator $(\id+\cR_p)^{i\tau}$ is bounded on $L^p(G)$.
Moreover 
$$
\exists C,\theta>0\quad
\forall \tau\in \bR\qquad
\|(\id+\cR_p)^{i\tau}\|_{\sL(L^p)} \leq C e^{\theta|\tau|}.
$$
and  for any $a\in \bC$, 
$$
\dom((\id+\cR_p)^a)=\dom ((\id+\cR_p)^{\Re a})
.
$$

\item 
\label{item_thm_fractional_power_commute}
For any $a,b\in \bC$,
the two (possibly unbounded) operators $\cR^a_p$ and $(\id+\cR_p)^b$ commute.
\item 
\label{item_thm_fractional_power_homogeneous}
For any $a\in \bC$, the operator $\cR^a_p$ is homogeneous of degree $\nu a$.

\end{enumerate}
\end{theorem}

Here we say that two (possibly unbounded) operators $A$ and $B$ commute when 
$$
x\in \dom (AB)\cap\dom (BA) \Longrightarrow ABx=BAx
.
$$
Let us recall that the domain of the product $AB$ of two (possibly unbounded) operators $A$ and $B$ on the same Banach space $\cX$ is formed by the elements $x\in \cX$ such that $x\in \dom (B)$ and $Bx\in \dom(A)$.

In Theorem~\ref{thm_fractional_power}
Part \eqref{item_thm_fractional_power_id_L_Gamma},
 $\Gamma$ denotes the usual Gamma function.

\begin{proof}
By Theorem \ref{thm_cRp} (i), the operator $\cR_p$ is closed and densely defined.
By Corollary \ref{cor_cRp_injective},
it is injective and Komatsu-non-negative in the sense that $(-\infty,0)$ is included in its resolvant set and it satisfies Property \eqref{eq_Komatsu_noneg}.
 Necessarily  $\id+\cR_p$ 
 also satisfies these properties.
 Furthermore $-(\id+\cR_p)$ generates an exponentially stable semigroup:
 $$
\| e^{-t(\id+\cR_p)}\|_{\sL(L^p(G))}
\leq e^{-t} \| e^{-t \cR_p}\|_{\sL(L^p(G))}
\leq e^{-t} \|h_1\|_1 ,
$$
and similarly for $C_o(G)$.

Most of the statements then follow from the general properties of fractional powers.
References for these results are in \cite{martinez+sanz_bk} as follows:
for Part \eqref{item_thm_fractional_power_common_cR_id+cR}
Corollaries 5.2.4, 5.1.12, and 5.1.13, together with
 and Section 7.1,
 for Parts \eqref{item_thm_fractional_power_id_L_Gamma}
 and
  \eqref{item_thm_fractional_power_id_I+L_Gamma}
Lemma 6.1.5,
 for Part \eqref{item_thm_fractional_power_cRitau}
 Section 7.1 and Corollary 7.1.2.
For Part \eqref{item_thm_fractional_domain_id+L_L},
 the property that the domains coincide comes from \cite[Theorem 5.1.7]{martinez+sanz_bk}.
 That they contain $\cS(G)$ and that the operators are unbounded is true for integer powers, hence true using Part  \eqref{item_thm_fractional_power_common_cR_id+cR}.
  
This concludes the proof of Theorem~\ref{thm_fractional_power},
except for Parts \eqref{item_thm_fractional_power_cRitau} and 
\eqref{item_thm_fractional_power_(id+cR)itau}.
For the moment, let us admit  that all the operators $\cR_p^{i\tau}$, $\tau\in \bR$,
are bounded.
Then for any $a\in \bC$, 
$\cR_p^{a+i\tau } $ is the closure of $\cR_p ^{i\tau}\cR_p^a$
by Part \eqref{item_thm_fractional_power_common_cR_id+cR},
so $\dom( \cR_p^{a+i\tau })  \supset \dom (\cR_p^a)$.
Applying this to $a\in \bR$ and $a-i\tau$, this shows
$\dom (\cR_p^{a'})=\dom (\cR_p^{\Re a'})$ for any $a'\in \bC$.
By Part \eqref{item_thm_fractional_domain_id+L_L},
we must also have
$\dom (\id+\cR_p)^{a'}=\dom (\id+\cR_p)^{\Re a'}$ for any $a'\in \bC$.
Now by Part \eqref{item_thm_fractional_power_common_cR_id+cR},
for any $\tau\in \bR$,
we have in the sense of operators
$$
(\id+\cR_p)^{i\tau} \supset (\id+\cR_p) (\id+\cR_p)^{-1+i\tau}
\quad\mbox{and}\quad 
\id \supset (\id+\cR_p)^{1-i\tau} (\id+\cR_p)^{-1+i\tau},
$$
hence 
$$
\range (\id+\cR_p)^{-1+i\tau} \subset \dom (\id+\cR_p)^{1-i\tau}  
=\dom  (\id+\cR_p),
$$
so
$$
\dom (\id+\cR_p)^{i\tau}
\supset
\dom (\id+\cR_p)^{-1+i\tau}.
$$
This last domain is the whole space $L^p$ since $(\id+\cR_p)^{-1+i\tau}$ is bounded by Part \eqref{item_thm_fractional_power_id_I+L_Gamma}.
Then the closed graph theorem\footnote{ref?}
implies that $(\id+\cR_p)^{i\tau}$ is bounded.
The bounds for the operator norms of 
$\cR_p^{i\tau}$ and 
$(\id+\cR_p)^{i\tau}$
comes from 
\cite[Proposition 8.1.1]{martinez+sanz_bk}.

Hence the proof of  Theorem \ref{thm_fractional_power} 
will be complete once we have proved that each  operator $\cR_p^{i\tau}$,
$\tau\in \bR$,  is bounded.
This will require a couple of technical lemmata.  
\end{proof}

\begin{lemma}
\label{lem_kernel_cRitau}
  Let $\cR$ be a positive Rockland operator on $G$,
  with heat kernels $h_t$.
For each $a\in \bC$ with $|\Re a|<1$ and $x\not=0$, 
the integral
\begin{equation}
\label{eq_lem_kernel_cRitau_kappaa}
\frac 1{\Gamma(\frac a \nu +1)}
\int_0^\infty t^{\frac a \nu} \cR h_t(x) dt
\end{equation}
is absolutely convergent for every $x\not=0$.
This defines kernel of type $a$.
\end{lemma}

\begin{proof}[Proof of Lemma~\ref{lem_kernel_cRitau}]
The estimates in \eqref{eq_control_heat_kernel_xfixed0}
 show the absolute convergence and the smoothness
of the function $\kappa_a$ on $G\backslash\{0\}$ given via 
\eqref{eq_lem_kernel_cRitau_kappaa}.
One checks easily, 
using an easy change of variable, 
 that $\kappa_a$ is 
homogeneous of degree $a-Q$.

Let us show that if $a=i\tau$ with $\tau\in \bR$, 
then $\kappa_a$ has mean average 0.
Let $b_1,b_2\in \bR$ with $0<b_1<b_2$.
We fix a homogeneous pseudo-norm $|\cdot|$ on $G$
and use the notation for the corresponding 
the polar change of coordinates
in Proposition \ref{prop_polar_coord}.
One can show easily that for $a=i\tau$
$$
\int_{b_1<|x|<b_2} 
\!\!\!\!\!\!
\kappa_a (x)dx
=
c_{\tau,b_1,b_2}
m_{\kappa_a}
\quad\mbox{where}\quad
c_{\tau,b_1,b_2}=\left\{\begin{array}{ll}
 \frac{b_{1}^{i\tau}-b_{2}^{i \tau}}{-i\tau}
&\mbox{if}\ \tau\not=0, \\
m_{\kappa_a} \ln \frac {b_{2}}{b_{1}}
&\mbox{if}\ \tau=0,
\end{array}\right.
$$
and $m_\tau=\int_\fS  \kappa_a d\sigma$ is the mean average of $\kappa_a$.
The properties of the heat kernel
(see Theorem \ref{thm_heat_kernel}) yield
\begin{eqnarray*}
&&\int_{b_1<|x|<b_2} 
\!\!\!\!\!\!
\kappa_a (x)dx
=
\frac 1{\Gamma(1+i\frac \tau \nu)}
\int_0^\infty t^{i\frac \tau \nu} 
\partial_t 
\int_{b_1<t^{\frac1\nu} |y|<b_2}
\!\!\!\!\!\!
\!\!\!\!\!\!
 h_1(y) dy\
 dt\\
 &&\quad=
 \frac 1{\Gamma(1+i\frac \tau \nu)}
\int_0^\infty t^{i\frac \tau \nu} 
 \frac  {-t^{-\frac 1\nu-1}}\nu
\left[
\int_\fS h_1(rx) d\sigma(x) r^{Q-1}
\right]_{r=t^{-\frac1\nu} b_1}^{t^{-\frac1\nu} b_2} 
dt.
\end{eqnarray*}
Moreover choosing e.g. $b_2=2b_1$ and $b_1=\ell^{-1}$ for every $\ell\in \bN$, we have
$|m_{\kappa_a}|\leq C_{\tau,h} \ell^{-Q+1}$,
hence $m_{\kappa_a}=0$.
Hence by Definition \ref{def_kernel_type_nu}
and Proposition \ref{prop_-Qhomo_distrib},
 $\kappa_a$ is a kerne of type $a$.
\end{proof}

We will also need the following technical result:

\begin{lemma}
\label{lem_cRN(cSG)}
We keep the setting and notation of Theorem~\ref{thm_cRp}.
For any $N\in \bN$,
the space $\cR^N(\cS(G))=\cR^N_p(\cS(G))$ is contained in $ \dom (\cR^N_p)\cap \range(\cR^N_p)$ and in $\cS(G)$.
The space $\cR^N_p(\cS(G))$ is  dense in $L^p(G)$ if $p\in (1,\infty)$,
and $\cR^N_{\infty_o}(\cS(G))$ is dense in $C_o(G)$.
\end{lemma}

\begin{proof}[Proof of Lemma~\ref{lem_cRN(cSG)}]
By Theorem~\ref{thm_cRp} Part (iv),
$$
\cS(G)\subset \dom (\cR^N_p)
\quad\mbox{and}\quad 
\cR^N_p(\cS(G))=\cR^N(\cS(G))\subset\cS(G)
.
$$

It remains to prove the properties of density.
For this we proceed as in the proof of Corollary~\ref{cor_cRp_injective}.
Let $\Psi$ be a bounded linear functional on $L^p(G)$ if $p\in (1,\infty)$
or on $C_o(G)$ if $p=\infty_o$ such that $\Psi$ vanishes identically on $\cR^N_p(\cS(G))$.
Realising $\Psi$ as the integration against a function $f\in L^{p'}(G)$ if $p\in (1,\infty)$ or $f\in M(G)$ if $p=\infty_o$, 
we have
 $\cR^N  f =0$.
Applying  Liouville's Theorem to $\cR^{N_1}$ (see Theorem \ref{thm_liouville_homogeneous_for_R}), 
this shows that 
$ f$ is a polynomial hence it must be identically zero.
This shows that $\Psi\equiv0$ and $\cR_p^N(\cS(G))$ is dense in $L^p(G)$.
\end{proof}

We can now prove
Part \eqref{item_thm_fractional_power_cRitau} of 
Theorem~\ref{thm_fractional_power}.

\begin{proof}[Proof of Part \eqref{item_thm_fractional_power_cRitau} of 
Theorem~\ref{thm_fractional_power}]
We keep the notation of Lemma~\ref{lem_kernel_cRitau}
and its proof.
The function $\kappa_a$ defined via \eqref{eq_lem_kernel_cRitau_kappaa}
is  smooth away from the origin,
$(a-Q)$-homogeneous, and with mean average 0 if $a\in i\bR$.
We assume $\Re a\in [0,Q)$.
By Theorem \ref{thm_op_type_Lp_bdd} and
Proposition \ref{prop_-Qhomo_distrib},
 the operator
$f\mapsto f*\kappa_a$ 
is bounded form $L^q(G)$ to $L^p(G)$ where
$\frac 1q-\frac 1p = \frac{\Re a} Q$,
$p,q\in (1,\infty)$.

By Lemma~\ref{lem_cRN(cSG)}, 
we can  apply the analyticity results of Theorem~\ref{thm_fractional_power}:
for $\phi\in \cR_p(\cS(G))$, 
$a\mapsto \cR_p^{-\frac a\nu}\phi$ is holomorphic on the strip 
$\{a\in \bC \ : \ |\Re \frac a\nu|<1\}$.
Furthermore, by Theorem \ref{thm_fractional_power} 
\eqref{item_thm_fractional_power_id_L_Gamma}, 
if $0<\Re \frac a\nu <1$, 
we have, with convergence in $L^p(G)$,
$$
\cR_p^{-\frac a\nu}\phi 
= \lim_{R\to\infty}
\frac 1{\Gamma(\frac a\nu)}
\int_0^R t^{\frac a \nu-1} \phi*h_t dt
.
$$
Integrating by parts, we get
\begin{equation}
\label{eq_pf_item_thm_fractional_power_cRitau}
\int_0^R t^{\frac a \nu-1} \phi*h_t  dt
=
\left[\frac {t^{\frac a\nu}}{\frac a\nu} 
\phi*h_t \right]_{t=0}^R
-
\int_0^R 
\frac {t^{\frac a\nu}}{\frac a\nu}
\partial_t \left(\phi*h_t \right) dt
.
\end{equation}
Now the first term (the bracket) at $t=0$ gives $0$ since $\Re \frac a\nu>0$
and since $\phi*h_{t}\to\phi$ in $L^{p}(G)$ as $t\to 0$ 
by Proposition \ref{prop_heat_kernel_semigroup_Lp}.
For the first term at $t=R$,
by the property of homogeneity of $h_t$ (see \eqref{eq_heat_kernel_homogeneity-1}),
we have
$$
\|\phi*h_t\|_p \leq  
t^{-\frac Q\nu}
\|h_1\circ D_{t^{-\frac 1\nu}}\|_p \|\phi\|_1
=
t^{-\frac Q\nu} t^{\frac Q{\nu p}}
\|h_1\|_p \|\phi\|_1
,
$$
so that the first term at $R$ gives $0$ as $R\to\infty$, for $\Re a<Q(1-\frac 1p).$
For the second term in the right-hand side of\eqref{eq_pf_item_thm_fractional_power_cRitau},
as $\partial_t \left(\phi*h_t \right) = \phi*\partial h_t =\phi *\cR h_t$,
we have 
$$
\frac 1{\Gamma(\frac a\nu)}\int_0^R 
\frac {t^{\frac a\nu}}{\frac a\nu}
\partial_t \left(\phi*h_t \right) dt
=
\frac 1{\Gamma(1+\frac a\nu)}\phi*\int_0^R t^{\frac a\nu}
\cR h_t dt
.
$$
By Lemma \ref{lem_kernel_cRitau}, if $\Re a<Q$,
this converges to
$\phi *\kappa_a$ in $L^p(G)$ as $R\to\infty$ since by the first part of this proof, we have
$$
\|\phi*\kappa_a\|_{L^p(G)}\leq C \|\phi \|_{L^q(G)} 
.
$$

We have obtained that 
for each $\phi \in \cR_p(\cS(G))$,
$a\mapsto \cR_p^{-\frac a\nu}\phi$ 
 is holomorphic on the strip 
$\{|\Re \frac a\nu|<1\}$
and coincides with $a\mapsto \phi *\kappa_a$
on $\{0<\Re a< Q(1-\frac 1p)\}$.
It is easy to check that $a\mapsto \phi *\kappa_a$ is
holomorphic on the strip $\{0<\Re a<Q\}$
and continuous on $\{0\leq \Re a<Q\}$.
This implies that, for $\Re a=0$, 
the closed operator
$\cR_p^{-\frac a\nu}$
and the bounded operator $\phi\mapsto  \phi*\kappa_a$
coincide on the dense subspace $\cR_p(\cS(G))$,
the latter convolution operator being bounded on $L^{p}(G)$ 
by Theorem \ref{thm_op_type_Lp_bdd} and
Proposition \ref{prop_-Qhomo_distrib}.
Thus for $\Re a=0$ the operator  $\cR_p^{-\frac a\nu}$ is bounded 
and is the convolution operator with kernel $\kappa_a$.
This concludes the proof of Part \eqref{item_thm_fractional_power_cRitau}
of  Theorem~\ref{thm_fractional_power}
and of the whole theorem.
\end{proof}

\begin{remark}
\label{rem_opnorm_I+R_itau}
The bound for  $\|\cR_p^{i\tau}\|_{\sL(L^p)}$ given in Theorem~\ref{thm_fractional_power}
\eqref{item_thm_fractional_power_cRitau}
may be improved by tracking down the various constants in the proof above and 
in the proof of Theorem \ref{thm_op_type_Lp_bdd}.
In the case of a sub-Laplacian, that is, $G$ stratified and $\cR=-\cL$,
using another argument,
Folland showed \cite[Proposition 3.14]{folland_75}
that 
$$
\|\cR_p^{i\tau}\|_{\sL(L^p)}\leq C_p |\Gamma (1-\tau)|^{-1}
\quad\mbox{and}\quad
\|(\id+\cR_p)^{i\tau}\|_{\sL(L^p)}\leq C_p |\Gamma (1-\tau)|^{-1},
$$
with $C_p$ independent of $\tau\in \bR$.
However his proof uses heavily the fact that the heat semigroup 
$\{e^{t\cL}\}_{t>0}=\{e^{-t\cR}\}_{t>0}$ is a strongly continuous semigroup of contractions 
preserving positivity,
see Remark \ref{rem_proof_thm_heat_kernel_schwartz}.
This can not be adapted in a simple way to the case of  a general Rockland operator.

We will not pursue the question of improving the bounds 
for $\|\cR_p^{i\tau}\|_{\sL(L^p)}$ and $\|(\id+\cR_p)^{i\tau}\|_{\sL(L^p)}$.
Indeed we will only need some bounds for  $\|(\id+\cR_p)^{i\tau}\|_{\sL(L^p)}$ 
 to show the property of interpolation between Sobolev spaces
(i.e. in the proof of Theorem \ref{thm_sobolev_spaces_interpolation}),
and the  bounds given in Theorem~\ref{thm_fractional_power} 
and later in Corollary \ref{cor_Riesz_Bessel_opnorm_power_I+R}
will be sufficient for our purpose.
\end{remark}

\subsection{Riesz and Bessel potentials}
\label{subsec_Riesz+Bessel_potentials}

In the next corollary, we proceed as in the proof of 
Theorem~\ref{thm_fractional_power}
\eqref{item_thm_fractional_power_cRitau}
 to realise 
$\cR_p^{-\frac a\nu}$ as a convolution operator with a homogeneous kernel smooth away from the origin
for certain values of $a$.
We also consider the left-invariant (but non-homogeneous) operator 
$(\id+\cR_p)^{-\frac a\nu}$.

\begin{definition}
Mimicking  the usual terminology in the Euclidean setting,
we call the operators $\cR^{-a/\nu}$ for  $\{a\in \bC, \ 0<\Re a<Q \}$
and $(\id+\cR)^{-a/\nu}$ for $a\in \bC_+$,
the \emph{Riesz potential} and the \emph{Bessel potential}, 
respectively.
In the sequel we will denote their kernels by $\cI_a$ and $\cB_a$, respectively,
as defined in the following:
\end{definition}

\begin{corollary}
\label{cor_Riesz_Bessel}
We keep the setting and notation of Theorem~\ref{thm_cRp}.

\begin{itemize}
\item[(i)]
Let $a\in \bC$ with $0<\Re a<Q$.
The integral 
$$
\cI_a(x):= \frac 1{\Gamma(a/\nu)}
\int_0^\infty t^{\frac{a}{\nu} -1} h_t(x) dt
,
$$
converges absolutely for every $x\not=0$.
This defines a distribution $\cI_a$ which is a kernel of type $a$, 
that is,
smooth away from the origin and $(a-Q)$-homogeneous.

For any $p\in (1,\infty)$,
if $\phi\in \cS(G)$ or, more generally, 
if $\phi\in L^q(G)\cap L^p(G)$
where $q\in [1,\infty)$ is given by
$\frac 1q -\frac 1p = \frac{\Re a} Q$,
then 
$$
\phi\in \dom (\cR_p^{-\frac a \nu})
\quad\mbox{and}\quad\cR_p^{- \frac a\nu} \phi =\phi*\cI_a \in L^p(G)
.
$$

\item[(ii)]
Let $a\in \bC_+$.
The integral 
$$
\cB_a(x):=
\frac 1{\Gamma(\frac a\nu)} \int_0^\infty t^{\frac a {\nu} -1} e^{-t} h_t(x) dt
,
$$
converges absolutely for every $x\not=0$.
The function $\cB_a$ is always smooth away from 0
and integrable on $G$.
If $\Re a> Q/2$, then $\cB_a\in L^2(G)$.

For each $a\in \bC_+$,
the operator $(\id+ \cR_p)^{-a/\nu}$ is a bounded
convolution operator on $L^p(G)$ for $p\in [1,\infty)$
or $C_o(G)$ for $p=\infty$,
 with the same (right convolution) kernel $\cB_a$.

If $a,b\in \cB_+$, then as integrable functions,
we have
$\cB_a*\cB_b=\cB_{a+b}$.
\end{itemize}
\end{corollary}

\begin{proof}[Proof of Corollary \ref{cor_Riesz_Bessel}]
The absolute convergence and 
the smoothness of $\cI_a$ and $\cB_a$ follow
from the estimates in \eqref{eq_control_heat_kernel_xfixed0}.

For the homogeneity of $\cI_a$, we use \eqref{eq_heat_kernel_homogeneity} and the change of variable $s=r^{-\nu} t$, to get
\begin{eqnarray*}
\cI_a(rx)
&=&\frac 1{\Gamma(a/\nu)}
\int_0^\infty t^{\frac{a}{\nu} -1} h_t(rx) dt
\\
&=&
\frac 1{\Gamma(a/\nu)}
\int_0^\infty (r^\nu s)^{\frac{a}{\nu} -1} r^{-Q} h_s(x) r^\nu ds
=
r^{a-Q} \cI_a(x)
.
\end{eqnarray*}

By Theorem \ref{thm_op_type_Lp_bdd},
the operator $\cS(G)\ni \phi\mapsto \phi*\cI_a$ 
is homogeneous of degree $-a$, and admits a bounded extension $L^q(G)\to L^p(G)$ 
when $\frac 1p-\frac1q=\frac {\Re (a)} Q$.
The rest of Part (i) follows from 
 Theorem~\ref{thm_fractional_power}
 together with Lemma \ref{lem_cRN(cSG)}.

By Theorem~\ref{thm_heat_kernel}, $\int_G |h_t|=1$ for all $t>0$,
so
\begin{equation}
\label{eq_L1norm_cBa}
\int_G |\cB_a(x)|dx 
\leq
\frac 1{|\Gamma(\frac a\nu)|} \int_0^\infty  t^{\frac{\Re a} {\nu} -1} e^{-t} 
\int_G |h_t(x)|dx \, dt
=
\frac {\Gamma(\frac {\Re a}\nu)}
{|\Gamma(\frac a\nu)|}
,
\end{equation}
and $\cB_a$ is integrable.
By Theorem~\ref{thm_fractional_power} Part
\eqref{item_thm_fractional_power_id_I+L_Gamma},
the integrable function $\cB_a$ is the convolution kernel of $(\id+\cR_p)^{-a/\nu}$.

Let us show the square integrability of $\cB_a$.
We assume $\Re a>0$.
We compute for any $R>0$:
\begin{eqnarray*}
&&\Gamma(a/\nu )^2\int_{|x|<R} |\cB_a(x)|^2 dx
=
\Gamma(a/\nu )^2\int_{|x|<R} \cB_a(x) \overline{\cB_a(x)} dx
\\
&&\quad=\int_{|x|<R}
\int_0^\infty t^{\frac a \nu -1} e^{-t} h_t(x) dt
\int_0^\infty  s^{\frac a \nu -1} e^{-s}  \bar h_s(x) ds
\,
dx
\\
&&\quad=
\int_0^\infty 
\int_0^\infty
(st)^{\frac a \nu -1} e^{-(t+s)} 
\int_{|x|<R} h_t(x)  \bar h_s(x) 
dx \, dt ds
.
\end{eqnarray*}
From the properties of the heat kernel 
(see \eqref{eq_homogeneity_sym+pos} and \eqref{eq_heat_kernel_hths_ht+s})
we see that
\begin{eqnarray*}
\int_{|x|<R} \!\!\!\!\!\!
h_t(x)  \bar h_s(x)dx
=
\int_{|x|<R} \!\!\!\!\!\!
h_t(x)  h_s(x^{-1})dx
\underset{R\to\infty}\longrightarrow
h_t*h_s(0),\\
\mbox{and}\
h_t*h_s(0)=h_{t+s}(0)=(t+s)^{-\frac Q\nu} h_1(0)
.
\end{eqnarray*}
Therefore,
\begin{eqnarray*}
&&\int_G |\cB_a(x)|^2 dx
=
\frac {h_1(0)}{\Gamma(a/\nu)^2} 
\int_0^\infty 
\int_0^\infty
(st)^{\frac a \nu -1} e^{-(t+s)} 
(t+s)^{-\frac Q \nu}
 dt ds
\\
&&\quad=
\frac {h_1(0)}{\Gamma(a/\nu)^2} 
\int_{s'=0}^1  \left(s'(1-s')\right)^{\frac a \nu -1} ds'
\int_{u=0}^\infty e^{-u} u^{2(\frac a \nu -1) - \frac Q \nu+1} du
,
\end{eqnarray*}
after the change of variables $u=s+t$ and $s'=s/u$.
The integrals over $s'$ and $u$ converge when $\Re a>Q/2$.
Thus $\cB_a$ is square integrable under this condition.
The rest of the proof of Corollary \ref{cor_Riesz_Bessel}
follows easily from the properties of 
the fractional powers of $\id+\cR$.
\end{proof}

\begin{corollary}
\label{cor_Riesz_Bessel_opnorm_power_I+R}
We keep the notation of Corollary \ref{cor_Riesz_Bessel}.
For any $a\in \bC_+$,
the operator norm of $(\id+\cR_p)^{-\frac a {\nu_{\cR}}}$
on $L^p(G)$ if $p\in [1,\infty)$ or $C_o(G)$ if $p=\infty_o$
is bounded by
$$
\|(\id+\cR_p)^{-\frac a {\nu_{\cR}}}\|_{\sL(L^{p})}
\leq \|\cB_{a}\|_1 \leq 
\Gamma\left(\frac{-\Re a}{\nu_{\cR}}\right)
\left|\Gamma\left(-\frac a{\nu_{\cR}}\right)\right|^{-1}.
$$

For any (fixed) $a\leq 0$ and $p\in (1,\infty)$, 
the following quantity is finite:
$$
\sup_{y\in \bR} e^{- 3|y|} \ln
\|(\id+\cR_{p})^{\frac{-a+iy }{\nu_{\cR}}}\|_{\sL(L^{p})}.
$$
\end{corollary}
\begin{proof}[Proof of Corollary \ref{cor_Riesz_Bessel_opnorm_power_I+R}]
The first part is a direct consequence of Corollary \ref{cor_Riesz_Bessel} (ii)
and its proof.
If $a>0$, the second part follows from the first 
 together with Sterling's estimates.
If $a=0$, it is a consequence of the exponential bounds for the operator norms obtained in   Theorem \ref{thm_fractional_power}
\eqref{item_thm_fractional_power_(id+cR)itau}.
\end{proof}

We now state the following technical lemma
and its corollaries
which will be useful in the sequel.

\begin{lemma}
\label{lem_cB_S}
We keep the notation of Corollary \ref{cor_Riesz_Bessel}.
\begin{itemize}
\item [(i)]
For any $\phi\in \cS(G)$ and $a\in \bC_+$, 
the function $\phi*\cB_a$ is Schwartz.
\item[(ii)]
Let  $a\in \bC$ and  $\phi\in \cS(G)$.
Then 
 $(\id +\cR_p)^a \phi$ does not depend on $p\in [1,\infty)\cup\{\infty_o\}$.
If $a\in \bN$,  $(\id +\cR_p)^a\phi$ coincides with $(\id+\cR)^a\phi$.
If $a\in \bC_+$,
we have
\begin{equation}
\label{eq_I+cR_phi*cB}
(\id+\cR_p)^a \left(\phi*\cB_{a\nu}\right) =
\left((\id+\cR_p)^a \phi\right)*\cB_{a\nu}
= \phi 
\qquad
(p\in [1,\infty)\cup\{\infty_o\}).
\end{equation}
\item[(iii)]
For any $N\in \bN$, $(\id+\cR)^N(\cS(G))=\cS(G)$.
\end{itemize}
\end{lemma}

\begin{proof}
Let $|\cdot|$ be a homogeneous pseudo-norm on $G$ and $N\in \bN$.
We see that
$$
\int_G|x|^N |\cB_a(x)|dx 
\leq
\frac 1{|\Gamma(\frac a\nu)|} \int_0^\infty  t^{\frac{\Re a} {\nu} -1} e^{-t} 
\int_G|x|^N  |h_t(x)|dx \, dt
,
$$
and using
the homogeneity of the heat kernel (see \eqref{eq_heat_kernel_homogeneity-1})
and the change of variables $y=t^{-\frac 1\nu}x$, we get
$$
\int_G|x|^N  |h_t(x)|dx
=
\int_G|t^{\frac 1\nu } y|^N  |h_1(y)|dy
=
c_N t^{\frac N\nu } 
,
$$
where $c_N=\| |y|^N h_1(y)\|_{L^1(dy)}$ is a finite constant 
since $h_1\in \cS(G)$.
Thus,
$$
\int_G|x|^N |\cB_a(x)|dx 
\leq
\frac {c_N}
{|\Gamma(\frac a\nu)|} \int_0^\infty  t^{\frac{\Re a} {\nu} -1 +\frac N\nu} e^{-t}  dt <\infty ,
$$
and $x\mapsto |x|^N \cB_a(x)$ is integrable.

Let $C_o\geq 1$ denote the constant in the triangle inequality for $|\cdot|$
(see Proposition~\ref{prop_homogeneous_pseudo_norm}).
Let also $\phi\in \cS(G)$.
We have for any  $N\in \bN$ and $\alpha\in \bN_0^n$:
\begin{eqnarray*}
&&(1+|x|)^N 
\left|\tilde X\left[ \phi*\cB_a\right](x)\right|
=
(1+|x|)^N 
\left|\tilde X \phi*\cB_a(x)\right|
\\&&\qquad\leq
(1+|x|)^N 
\left|\tilde X \phi\right|*\left|\cB_a\right|(x)
\\
&&\qquad\leq
C_o^{N}
\left|(1+|\cdot|)^N\tilde X \phi\right|*
\left|(1+|\cdot|)^N\cB_a(x)\right|(x)
\\
&&\qquad\leq
C_o^{N}
\left\|(1+|\cdot|)^N\tilde X \phi\right\|_\infty
\left\|(1+|\cdot|)^N\cB_a\right\|_{L^1(G)}
.
\end{eqnarray*}
This shows that that $\phi*\cB_a\in \cS(G)$
 (for a description of the Schwartz class, see 
 \cite[Chapter 1 D]{folland+stein_bk82}
 and Part (i) is proved.

Part (ii) follows easily from 
Theorem~\ref{thm_fractional_power} and Corollary \ref{cor_Riesz_Bessel}.

Let us prove Part (iii).
By Theorem~\ref{thm_cRp} (iv), we have the inclusion 
$(\id+\cR)^N(\cS(G))\subset\cS(G)$.
The reverse inclusion $\cS(G)\subset (\id+\cR)^N(\cS(G))$
follows from  \eqref{eq_I+cR_phi*cB}
and  Theorem~\ref{thm_cRp} (iv).
So for any $N\in \bN$, $\cS(G)$ is included in  
$\dom\left[(\id+\cR_p)^N\right]\cap \range \left[(\id+\cR_p)^N\right]$
and we can apply 
the analyticity results of Theorem~\ref{thm_fractional_power}:
 the function
$a \mapsto (\id+\cR_p)^a \phi$ is holomorphic 
in $\{a\in \bC:\, -N < \Re a< N\}$.
We observe that by Corollary \ref{cor_Riesz_Bessel} (ii), 
if $-N<\Re a<0$,
all the functions 
$ (\id+\cR_p)^a \phi$ coincide with $\phi*\cB_{a\nu}$ 
 for any $p\in [1,\infty)\cup\{\infty_o\}$.
This shows that for each $a\in \bC$ fixed, 
$(\id +\cR_p)^a \phi$ is independent of~$p$.
This concludes the proof of Lemma \ref{lem_cB_S}.
\end{proof}

 \section{Sobolev spaces on graded groups}
 \label{SEC:Sobolev}
 
  In this section 
 we define the Sobolev spaces 
 associated to a positive Rockland operator $\cR$
 and  show that they satisfy similar properties to the Euclidean Sobolev spaces.
 We will show that the constructed spaces are actually independent of
 the choice of a positive Rockland operator $\cR$ on a graded group with which
 we start our construction.

 \subsection{Definition and first properties of Sobolev spaces}

We first need the following lemma:
\begin{lemma}
\label{lem_pre_sobolev}
We keep the notation of Theorem~\ref{thm_fractional_power}.
For any $s\in \bR$ and $p\in [1,\infty)\cup\{\infty_o\}$,
the domain of the operator $(\id+\cR_p)^{\frac s\nu}$ 
contains $\cS(G)$,
and the map
$$
f\longmapsto
\|(\id+\cR_p)^{\frac s\nu} f\|_{L^p(G)}
$$
defines a norm on $\cS(G)$. We denote it by
$$\|f\|_{L^{p}_{s}(G)}:=\|(\id+\cR_p)^{\frac s\nu} f\|_{L^p(G)}.$$
Moreover,
 any sequence in $\cS(G)$ which is Cauchy for $\|\cdot\|_{L^p_s(G)}$ is convergent in $\cS'(G)$.
\end{lemma}

We have allowed ourselves to write 
$\|\cdot\|_{L^\infty(G)}=\|\cdot\|_{L^{\infty_o}(G)}$
for the supremum norm.
We may also write $\|\cdot\|_\infty$ or $\|\cdot\|_{\infty_o}$.

\begin{proof}
The domain of $(\id+\cR_p)^{\frac s\nu}$ contains $\cS(G)$:
\begin{itemize}
\item 
by Theorem~\ref{thm_fractional_power} 
Part \eqref{item_thm_fractional_domain_id+L_L} for $s>0$, 
\item 
by Corollary \ref{cor_Riesz_Bessel} (ii) for $s<0$ and,
\item 
 for $s=0$, since $(\id+\cR_p)^{\frac s\nu}=\id$.
\end{itemize}
Since the operator $(\id+\cR_p)^{\frac s\nu}$ is linear,
it is easy to check that the map
$f\mapsto \|(\id+\cR_p)^{\frac s\nu} f\|_p$ 
is non-negative and satisfies the triangle inequality.
Since $(\id+\cR_p)^{s/\nu}$ is injective
by Theorem~\ref{thm_fractional_power}, Part \eqref{item_thm_fractional_power_common_cR_id+cR},
we have that
 $ \|f\|_{L^p_s(G)}=0$ implies $f=0$.

Clearly $\|\cdot\|_{L^p_0(G)}=\|\cdot\|_p$, so in the case of $s=0$ a Cauchy sequence 
of Schwartz functions converges in $L^p$-norm, thus also in $\cS'(G)$.

Let us assume $s>0$.
By Corollary \ref{cor_Riesz_Bessel} (ii),
 the operator $(\id+\cR_p)^{-\frac s\nu}$ is bounded on $L^p(G)$.
 Hence we have $\|\cdot\|_{L^p(G)} \leq C\|\cdot\|_{L^p_s(G)}$ on $\cS(G)$.
Consequently a $\|\cdot\|_{L^p_s(G)}$-Cauchy sequence 
of Schwartz functions converge in $L^p$-norm thus in $\cS'(G)$.

Now let us assume $s<0$.
Let $\{f_\ell\}_{\ell\in \bN}$ be a sequence of Schwartz functions  which is Cauchy for the norm $\|\cdot\|_{L^p_s(G)}$.
By \eqref{eq_I+cR_phi*cB}
we have 
$f_\ell=
\left((\id+\cR_p)^{\frac s\nu} f_\ell\right)*\cB_{s}$.
Furthermore, if $\phi\in \cS(G)$ then using \eqref{eq_int_f*gh_f} 
and \eqref{eq_homogeneity_sym+pos},
we have
\begin{equation}
\label{eq_pf_lem_pre_sobolev}
\int_G f_\ell(x) \phi(x) dx =
\int_G \left((\id+\cR_p)^{\frac s\nu} f_\ell\right)(x) \  \left(\phi* \cB_{s}\right) (x) \ dx
 . 
\end{equation}
By assumption the sequence
$\{(\id+\cR_p)^{\frac s\nu} f_\ell\}_{\ell\in \bN}$ is 
$\|\cdot\|_{L^p(G)}$-Cauchy thus convergent in $L^p(G)$. 
By Lemma~\ref{lem_cB_S},  $\phi*\cB_s\in \cS(G)$.
Therefore, the right hand-side of \eqref{eq_pf_lem_pre_sobolev} 
 is convergent as $\ell\to\infty$.
Hence the scalar sequence $\langle f_\ell, \phi\rangle$ converges for any $\phi\in \cS(G)$.
This shows that the sequence $\{f_\ell\}$ converges in $\cS'(G)$.
\end{proof} 
 
Lemma~\ref{lem_pre_sobolev} allows us to define the Sobolev spaces: 
 \begin{definition}
 \label{def_sobolev_space}
 \index{Sobolev space $L^p_s(G)$}
 Let $\cR$ be a positive Rockland operator on $G$.
We consider its $L^p$-analogue $\cR_p$ 
and the powers of $(\id+\cR_p)^a$
as defined in Theorems \ref{thm_cRp} and \ref{thm_fractional_power}.
Let $s\in\bR.$

If $p\in [1,\infty)$,
the \emph{Sobolev space} $L^p_{s,\cR}(G)$ 
is the subspace of $\cS'(G)$
obtained by completion 
of $\cS(G)$ 
with respect to the \emph{Sobolev norm}
$$
\|f\|_{L^p_{s,\cR} (G)}:=\|(\id+\cR_p)^{\frac s\nu} f\|_{L^p(G)}
,\quad f\in \cS(G) .
$$

If $p=\infty_o$, 
the \emph{Sobolev space} $L^{\infty_o}_{s,\cR}(G)$ 
is the subspace of $\cS'(G)$
obtained by completion 
of $\cS(G)$ 
with respect to the \emph{Sobolev norm}
$$
\|f\|_{L^{\infty_o}_{s,\cR}(G)}:=\|(\id+\cR_{\infty_o})^{\frac s\nu} f\|_{L^\infty(G)}
,\quad f\in \cS(G) .
$$

When the Rockland operator $\cR$ is fixed, we may allow ourselves to drop the index $\cR$ in $L^p_{s,\cR}(G)= L^p_s(G)$ to simplify the notation.
\end{definition}

We will see later that the Sobolev spaces
do not depend on the Rockland operator $\cR$,
see Theorem \ref{thm_Lps_indep_cR}.

 By construction
the Sobolev space $L^p_s(G)$ endowed with the Sobolev norm 
is a Banach space
which contains $\cS(G)$ as a dense subspace
and is included in $\cS'(G)$.
The Sobolev spaces share many properties with their Euclidean counterparts.

\begin{theorem}
\label{thm_sobolev_spaces}
 Let $\cR$ be a positive Rockland operator on $G$.
We consider the associated Sobolev spaces $L^p_s(G)$
for $p\in [1,\infty)\cup\{\infty_o\}$ and $s\in \bR$.

\begin{enumerate}
\item
\label{item_thm_sobolev_spaces_s=0}
If $s=0$, then
 $L^p_0(G)=L^p(G)$ for $p\in [1,\infty)$
with $\|\cdot\|_{L^p_0(G)}=\|\cdot\|_{L^p(G)}$,
 and 
 $L^{\infty_o}_0(G)=C_o(G)$ 
with $\|\cdot\|_{L^{\infty_o}_0(G)}=\|\cdot\|_{L^\infty(G)}$.
\item
\label{item_thm_sobolev_spaces_s>0}
 If $s>0$, then for any $a\in \bC$ with $\Re a=s$, we have
$$
L^p_s(G)=
\dom\left[ (\id+\cR_p)^{\frac a\nu}\right]=\dom(\cR_p^{\frac a\nu})
\subsetneq L^p(G),
$$ 
and 
the following norms are equivalent to $\|\cdot\|_{L^p_s(G)}$: 
$$
f\longmapsto \|f\|_{L^p(G)} +\|(\id+\cR_p)^{\frac s \nu} f\|_{L^p(G)}
,\ 
f\longmapsto \|f\|_{L^p(G)} +\|\cR_p^{\frac s\nu} f\|_{L^p(G)}
.
$$

\item 
\label{item_thm_sobolev_spaces_distrib}
Let $s\in \bR$ and $f\in \cS'(G)$.
\begin{itemize}
\item 
Given  $p\in (1,\infty)$, then
$f\in L^p_s(G)$ if and only if $(\id+\cR_p)^{s/\nu} f \in L^p(G)$ 
in the sense that 
the  linear mapping
$\cS(G)\ni \phi \mapsto \langle f, (\id+\bar \cR_{p'})^{s/\nu} \phi\rangle$
extends to a bounded functional on $L^{p'}(G)$ where 
$p'$ is the conjugate exponent of $p$.

\item 
$f\in L^1_s(G)$ 
if and only if 
 $(\id+\cR_1)^{s/\nu} f \in L^1(G)$ 
in the sense that
the  linear mapping
$\cS(G)\ni \phi \mapsto \langle f, (\id+\bar\cR_{\infty_o})^{s/\nu} \phi\rangle$
 extends to a bounded functional on  $C_o(G)$ 
 and is realised as a measure given by an integrable function.

\item 
$f\in L^{\infty_o}_s(G)$ 
if and only if $(\id+\cR_1)^{s/\nu} f \in C_o(G)$ 
in the sense that
the  linear mapping
$\cS(G)\ni \phi \mapsto \langle f, (\id+\bar\cR_1)^{s/\nu} \phi\rangle$
extends to a bounded functional on  $L^1(G)$ 
and is realised as integration against function in $C_o(G)$
\end{itemize}

\item 
\label{item_thm_sobolev_spaces_inclusions}
If $a,b\in \bR$ with $a< b$ and $p\in [1,\infty)\cup\{\infty_o\}$, 
then the following continuous strict inclusions hold
$$
\cS(G)\subsetneq L^p_b (G)\subsetneq L^p_a(G) \subsetneq \cS'(G)
,
$$
and an equivalent norm for $L^p_b(G)$ is
$$
L^p_b(G)\ni f \longmapsto \|f\|_{L^p_a(G)} + \|\cR ^{\frac {b-a} \nu} f\|_{L^p_a(G)}
.
$$

\end{enumerate}
\end{theorem}

From now on, we will often use the notation $L^p_0(G)$
since this allows us not to distinguish between the cases
$L^p_0(G)=L^p(G)$ when $p\in [1,\infty)$
and 
$L^p_0(G)=C_o(G)$ when $p=\infty_o$.

 \begin{proof}[Proof of  Theorem~\ref{thm_sobolev_spaces}]
 Part \eqref{item_thm_sobolev_spaces_s=0} is true since 
$(\id+\cR_p)^{\frac 0\nu}=\id$.
  
  Let us prove 
 Part \eqref{item_thm_sobolev_spaces_s>0}.  
So let $s>0$.
Clearly  $L^p_s(G)$ coincides with the domain 
of the unbounded operator $(\id+\cR_p)^{\frac s\nu}$
(see  Theorem~\ref{thm_fractional_power} \eqref{item_thm_fractional_domain_id+L_L}) hence
it is a proper subspace of $L^p(G)$.
As
 the operator $(\id+\cR_p)^{-\frac s\nu}$ is bounded on $L^p(G)$,
we have $\|\cdot\|_{L^p(G)} \leq C\|\cdot\|_{L^p_s(G)}$ on $L^p_s(G)$.
So $\|\cdot\|_{L^p(G)} +\|\cdot\|_{L^p_s(G)}$ 
is a norm on $L^p_s(G)$ which is equivalent to  the Sobolev norm.
By Theorem \ref{thm_fractional_power},
the operators $\cR_p^{\frac s\nu}$ and $(\id+\cR_p)^{\frac s\nu}$
 share the same domain.
 Hence  Part \eqref{item_thm_sobolev_spaces_s>0} follows from general functional analysis, especially the closed graph theorem.

Part \eqref{item_thm_sobolev_spaces_distrib}
follows from Part \eqref{item_thm_sobolev_spaces_s>0} in the case $s\geq 0$.
We now consider the case $s<0$.
By Lemma~\ref{lem_cB_S} and Corollary \ref{cor_Riesz_Bessel},
the mapping
$$
T_{s,p',f}:\cS(G)\ni \phi \longmapsto 
\langle f, (\id+\bar \cR_{p'})^{s/\nu} \phi \rangle
= 
\langle f, \phi*\bar \cB_{-s} \rangle
$$ 
is well defined for any $f\in \cS'(G)$.
If $T_{s,p',f}$ admits a bounded extension to a functional on $L^{p'}_0(G)$,
then we denote this extension $\tilde T_{s,p',f}$ and we have 
$\|\tilde T_{s,p',f}\|_{\sL (L^{p'})}= \|f\|_{L^p_s(G)}$.
This is certainly so if $f\in \cS(G)$.
The proof of Part \eqref{item_thm_sobolev_spaces_distrib}
follows from the following observation:
a sequence $\{f_\ell\}_{\ell\in \bN}$ of Schwartz functions
is convergent for the Sobolev norm $\|\cdot\|_{L^p_s(G)}$ if and only if 
 $\{\tilde T_{s,p',f_\ell}\}$ is convergent  in $L^{p'}_0(G)$.

Let us show Part \eqref{item_thm_sobolev_spaces_inclusions}.
Let $a\leq b$ and $p\in [1,\infty)\cup\{\infty_o\}$.
Theorem~\ref{thm_fractional_power} implies that 
  for any $f\in \cS(G)$ we have
$$
\|f\|_{L^p_b(G)}
\leq
 \|(\id+\cR_p)^{\frac {b-a}\nu}\|_{\sL (L^p_0)}\|f\|_{L^p_a}.
$$ By density of $\cS(G)$, this yields
the continuous inclusion $L^p_b\subset L^p_a$.
If $a<b$, we also have
$$
\|f\|_{L^p_b(G)}
=
 \|(\id+\cR_p)^{\frac a\nu} f\|_{L^p_{b-a}(G)}
\asymp
\|(\id+\cR_p)^{\frac a\nu} f\|_{L^p(G)}
+\|\cR_p^{\frac{b-a}\nu}(\id+\cR_p)^{\frac a\nu} f\|_{L^p(G)}
$$
by Part \eqref{item_thm_sobolev_spaces_s>0} above for any $f\in \cS(G)$.
By Theorem~\ref{thm_fractional_power} 
\eqref{item_thm_fractional_power_commute},
we can commute the operators 
$\cR_p^{\frac{b-a}\nu}$ and $(\id+\cR_p)^{\frac a\nu}$ in this last expression.
This shows that the Sobolev norm is equivalent to 
$ \|\cdot\|_{L^p_b(G)}
\asymp
\|\cdot\|_{L^p_a(G)}
+\|\cR_p^{\frac{b-a}\nu} \cdot\|_{L^p_a(G)}$.
Since the operator $\cR_p^{\frac{b-a}\nu}$ is unbounded, this also implies the strict inclusions
given in Part \eqref{item_thm_sobolev_spaces_inclusions}.
This concludes the proof of this part and of the whole theorem.
\end{proof}

Theorem~\ref{thm_sobolev_spaces} has the following corollaries.
The first two are easy consequences of Part \eqref{item_thm_sobolev_spaces_distrib} left to the reader.

\begin{corollary}
We keep the setting and notation of Theorem~\ref{thm_sobolev_spaces}.
Let $s<0$ and $p\in [1,\infty)\cup\{\infty_o\}$. 
Let $f\in \cS'(G)$.

The tempered distribution $f$ is in $L^p_s(G)$ if and only if 
the mapping $\phi\in \cS(G) \mapsto \langle f, \phi * \bar\cB_{-s}\rangle$ 
extends to a bounded linear functional on $L^{p'}_0(G)$
with the additional property  that
\begin{itemize}
\item 
for $p=1$,
this functional on $C_o(G)$ is realised as a measure given by an integrable function,
\item 
if $p=\infty_o$,
this functional on $L^1(G)$ is realised 
by integration against a function in $C_o(G)$.
 \end{itemize}
\end{corollary}

 \begin{corollary}
\label{cor_sobolev_spaces_duality}
 Let $\cR$ be a positive Rockland operator on a graded Lie group $G$.
We consider the associated Sobolev spaces $L^p_{s,\cR}(G)$.
If  $s\in \bR$ and $p\in (1,\infty)$,
the dual space of $L^p_{s,\cR}(G)$ is isomorphic to $L^{p'}_{-s,\bar\cR}(G)$ 
via the distributional duality,
where 
$p'$ is the conjugate exponent of $p$.
\end{corollary}

Corollary \ref{cor_sobolev_spaces_duality} will be improved in Proposition \ref{prop_sobolev_spaces_duality} once we show (see Theorem \ref{thm_Lps_indep_cR}) that Sobolev spaces are indeed independent of the considered Rockland operator.

\begin{corollary}
\label{cor_sobolev_spaces_cD_dense}
We keep the setting and notation of Theorem~\ref{thm_sobolev_spaces}.
Let $s\in \bR$ and $p\in [1,\infty)\cup\{\infty_o\}$. 
Then $\cD(G)$ is dense in $L^p_s(G)$.
\end{corollary}
\begin{proof}[Proof of Corollary~\ref{cor_sobolev_spaces_cD_dense}]
This is certainly true for $s\geq 0$ (see the proof of Parts \eqref{item_thm_sobolev_spaces_s=0} 
and  \eqref{item_thm_sobolev_spaces_s>0} 
of Theorem~\ref{thm_sobolev_spaces}).
For $s<0$, it suffices to proceed as in the last part of the proof of 
Part~\eqref{item_thm_sobolev_spaces_distrib} 
with a sequence of functions $f_\ell\in \cD(G)$.
\end{proof}

\medskip

 In the next statement, we show how to produce functions and converging sequences in Sobolev spaces
 using the convolution:
\begin{proposition}
\label{prop_construction_fcn_sobolov_spaces}
We keep the setting and notation of Theorem~\ref{thm_sobolev_spaces}.
Here $a \in \bR$ and $p\in [1,\infty)\cup\{\infty_o\}$.
\begin{itemize}
\item[(i)] 
If $f\in L^p_0(G)$ and  $\phi\in \cS(G)$, then $f*\phi\in L^p_a$ for any $a$ and $p$.

\item[(ii)] 
If $f\in L^p_a(G)$ and $\psi\in \cS(G)$, 
then 
\begin{equation}
\label{eq_I+cR_psi*f}
(\id+\cR_p)^{\frac a\nu} (\psi*f) = 
 \psi*\left((\id+\cR_p)^{\frac a\nu}f\right) ,
\end{equation}
and $\psi*f\in L^p_a(G)$
with $\| \psi*f \|_{L^p_a(G)}\leq \|\psi\|_{L^1(G)} \| f \|_{L^p_a(G)}$.
Furthermore, writing 
$ \psi_\epsilon(x):= \epsilon^{-Q} \psi(\epsilon^{-1} x) $
for each $\epsilon>0$, 
then $\{\psi_\epsilon* f\}$ converges to  $f$ in $L^p_a(G)$ as $\epsilon\to0$.
\end{itemize}
\end{proposition}

\begin{proof}[Proof of Proposition~\ref{prop_construction_fcn_sobolov_spaces}]
Let us prove Part (i).
Here $f\in L^p_0(G)$. 
By density of $\cS(G)$ in $L^p_0(G)$, we can find 
a sequence of Schwartz functions $\{f_\ell\}$ 
converging to $f$ in $L^p_0$-norm.
Then $f_\ell*\phi\in \cS(G)$ and for any $N\in \bN$, 
$$
\cR^N(f_\ell*\phi)=f_\ell*\cR^N \phi
\underset{\ell\to\infty} \longrightarrow f*\cR^N \phi
\quad\mbox{in}\ L^p_0(G)
,
$$
thus $\cR^N_p (f*\phi)=f*\cR^N \phi\in L^p(G)$ and 
$$
\|f*\phi\|_{L^p_0(G)} + \|\cR ^N_p (f*\phi)\|_{L^p_0(G)}
\;<\infty .
$$
By Theorem~\ref{thm_sobolev_spaces} \eqref{item_thm_sobolev_spaces_inclusions},
  this shows that $f*\phi$ is in $L^p_{\nu N}$ for any $N\in \bN$,
  hence in any $p$-Sobolev spaces
  (cf. Theorem \ref{thm_sobolev_spaces} 
  \eqref{item_thm_sobolev_spaces_inclusions}). 
  This proves (i).

Let us prove Part (ii).
We observe that 
both sides of Formula \eqref{eq_I+cR_psi*f}
 always make sense as  convolutions of a Schwartz function with a tempered distribution.
Formula \eqref{eq_I+cR_psi*f}
 is clearly true if $a<0$ by Corollary \ref{cor_Riesz_Bessel} (ii)
since then the $(\id+\cR_p)^{\frac a\nu}$ is a convolution operator.
Consequently 
\eqref{eq_I+cR_psi*f} is true also for any $f,\psi\in \cS(G)$ and $a\in \bR$
by the analyticity result of Theorem~\ref{thm_fractional_power}
and Lemma \ref{lem_cB_S}. 
Using this result for Schwartz functions yields
 that Equality \eqref{eq_I+cR_psi*f}  holds as distributions  for any $f\in L^p_a(G)$,
$\phi\in \cS(G)$, and $a\in \bR$, since we have
\begin{eqnarray*}
&&\langle (\id+\cR_p)^{\frac a\nu} (\psi*f) , \phi\rangle
=\langle \psi*f ,  (\id+ \bar\cR_{p'})^{\frac a\nu}\phi\rangle
=\langle f ,  \tilde \psi* (\id+ \bar\cR_{p'})^{\frac a\nu}\phi\rangle
\\&&\qquad
=\langle f , (\id+ \bar\cR_{p'})^{\frac a\nu} ( \tilde \psi* \phi)\rangle
=\langle (\id+ \cR_p)^{\frac a\nu} f , \tilde \psi* \phi \rangle
.
\end{eqnarray*}
Taking the $L^p$-norm on both sides of Equality \eqref{eq_I+cR_psi*f}  yields
$$
\|(\id+\cR_p)^{\frac a\nu} (\psi*f)\|_p = 
\| \psi*\left((\id+\cR_p)^{\frac a\nu}f\right) \|_p
\leq \|\psi\|_1 
\| (\id+\cR_p)^{\frac a\nu}f \|_p.
$$
Hence $\psi*f\in L^p_a(G)$ with $L^p_a$-norm $\leq \|\psi\|_1 
\| f \|_{L^p_a(G)}$.
Moreover, by Lemma~\ref{lem_approximatation},
\begin{eqnarray*}
&&\|\psi_\epsilon*f -f\|_{L^p_a(G)}
=
\| (\id+ \cR_p)^{\frac a\nu}(\psi_\epsilon*f -f)\|_p
\\&&\qquad=
\| \psi_\epsilon* \left((\id+ \cR_p)^{\frac a\nu}f\right) - (\id+ \cR_p)^{\frac a\nu}f\|_p
\longrightarrow_{\epsilon\to0}0, 
\end{eqnarray*}
that is, $\{\psi_\epsilon*f \}$ converges to $f$ in $L^p_a(G)$ as $\epsilon\to 0$.
This proves (ii).
\end{proof}

We note that, in general, 
keeping the notation of Proposition~\ref{prop_construction_fcn_sobolov_spaces} (ii),
it is not possible to prove that 
$\{ f*\psi_\epsilon\}$ converges to  $f$ in $L^p_a(G)$ as $\epsilon\to0$
for any sequence $\{\psi_\epsilon\}_{\epsilon>0}$.
We need to know that $\{\psi_\epsilon\}_{\epsilon>0}$
yields an $L^p$ approximation of the identity, 
that is, $f*\psi_\epsilon\to f$ in $L^p$ as $\epsilon\to0$ for any $f\in L^p(G)$.

\subsection{Interpolation between Sobolev spaces}
\label{sec_interpolation_Lps}
\index{interpolation!between Sobolev spaces}

In this section, we prove that interpolation between Sobolev spaces $L^p_a(G)$ works in the same way as its Euclidean counterpart.

\begin{theorem}
\label{thm_sobolev_spaces_interpolation}
Let $\cR$ and $\cQ$ be two positive Rockland operators 
on two graded Lie groups $G$ and $F$.
We consider their associated Sobolev spaces $L^p_a(G)$ and $L^q_b(F)$.
Let $p_0,p_1,q_0,q_1\in (1,\infty)$ 
and real numbers $a_0,a_1,b_0,b_1$.

We also consider a linear mapping $T$ from $L^{p_0}_{a_0}(G) + L^{p_1}_{a_1}(G)$
to locally integrable functions on $F$.
We assume that $T$ maps $L^{p_0}_{a_0}(G)$ and $L^{p_1}_{a_1}(G)$
boundedly into $L^{q_0}_{b_0}(F)$ and $L^{q_1}_{b_1}(F)$, respectively.

Then $T$ extends uniquely to a bounded mapping from 
$L^p_{a_t}(G)$ to $L^q_{b_t}(F)$ for $t\in [0,1]$
where $a_t,b_t,p_t,q_t$ are defined by
$$
\left(a_t,b_t,\frac 1{p_t},\frac 1{q_t}\right) = (1-t)\left(a_0,b_0,\frac 1{p_0},\frac 1{q_0}\right)
+t\left(a_1,b_1,\frac1{p_1},\frac 1{q_1}\right)
.
$$
\end{theorem}

The idea of the proof is similar to the one of the Euclidean or stratified cases,
see \cite[Theorem 4.7]{folland_75},
with some  modifications 
since our estimates for $\|(\id+\cR)^{i\tau}\|_{\sL(L^p)}$ are different to 
the ones obtained by Folland in \cite{folland_75}.
For this,
compare Corollary \ref{cor_Riesz_Bessel_opnorm_power_I+R} in this monograph with \cite[Proposition 4.3]{folland_75}.
See also Remark \ref{rem_opnorm_I+R_itau}.

\begin{proof}[Proof of Theorem~\ref{thm_sobolev_spaces_interpolation}]
By duality (see Corollary \ref{cor_sobolev_spaces_duality})
and up to a change of notation, 
 it suffices to prove  the case 
$a_1\geq a_0$ and $b_1\leq b_0$.
The idea is to interpolate between the operators 
formally given by
\begin{equation}
\label{eq_def_Tz_pf_thm_sobolev_spaces_interpolation}
T_z= (\id+\cQ)^{b_z/\nu_\cQ} T (\id+\cR)^{-a_z/\nu_\cR},
\end{equation}
where $\nu_\cR$ and $\nu_\cQ$ denote the degrees of homogeneity of $\cR$ and $\cQ$ respectively
and the complex numbers $a_z$ and $b_z$ are defined by
$$
\left(a_z,b_z\right) := z\left(a_1,b_1\right)
+(1-z) \left(a_0,b_0\right),
$$
for $z$ in the strip
$$
S:=\{z\in \bC \ : \ \Re z\in [0,1]\}
.
$$
In \eqref{eq_def_Tz_pf_thm_sobolev_spaces_interpolation},
we have abused the notation regarding the fractional powers of $\cR_p$ and $\cQ_q$ and removed $p$ and $q$. 
This is possible by Lemma  \ref{lem_cB_S} and density of the Schwartz space in each Sobolev space. Hence \eqref{eq_def_Tz_pf_thm_sobolev_spaces_interpolation} makes sense. 

By Lemma \ref{lem_cB_S},
for any $\phi\in \cS(G)$ and $\psi\in \cS(F)$, 
we have
$$
\langle T_z \phi , \psi \rangle 
=
\langle T (\id+\cR)^{-N -\frac{a_z}{\nu_\cR}}
(\id+\cR)^{N} \phi , 
(\id+\bar \cQ)^{-M +\frac{ b_z} {\nu_\cQ}} 
(\id+\bar \cQ)^{M} 
 \psi \rangle 
$$
for any $M,N\in \bZ$.
In particular for $-M$ and $-N$ large enough, 
 Theorem \ref{thm_fractional_power} implies
 that $S\ni z\mapsto \langle T_z \phi , \psi \rangle $
is analytic.
With $M=N\in \bN$ the smallest integer with $N>a_1,a_0, b_1,b_0$, 
by Corollary \ref{cor_Riesz_Bessel_opnorm_power_I+R},
we get
$$
\left|\langle T_z \phi , \psi \rangle \right| \leq 
\frac{\Gamma\left(\frac{N -\Re z (a_1-a_0)}{\nu_{\cR}}\right)}
{\left|\Gamma\left(\frac {N-z(a_1-a_0)}{\nu_{\cR}}\right)\right|}
\frac{\Gamma\left(\frac{N-\Re z(b_0-b_1)}{\nu_{\cQ}}\right)}
{\left|\Gamma\left(\frac {N-z(b_0-b_1)}{\nu_{\cQ}}\right)\right|}
\|T\|_{\sL(L^{p_1}_{a_1}, L^{q_1}_{b_1})}
 \|\phi\|_{L^{p_1}_{N}}\|\psi \|_{L^{q_1'}_{N}}
.
$$
Using Sterling's estimates,
we obtain
$$
\forall z=x+iy\in S\qquad
\ln \left|\langle T_z \phi , \psi \rangle \right| \leq 
  \ln |y| (2|y|+O(\ln |y|) )
 $$
with the constant from the notation $O$ depending on $\phi,\psi,a_1,a_0,b_1,b_0$.

The operator norms of 
$T_z$ for $z$ on the boundary of the strip, that is, $z=j+iy$, $j=0,1$, $y\in \bR$
may be easily estimated by:
$$
\|T_z \|_{\cL(L^{p_j}, L^{q_j}) }
 \leq
\|(\id+\cQ_{q_j})^{\frac{b_z-b_j}{\nu_{\cQ}}}\|_{\sL(L^{q_j})}
\|T\|_{\sL(L^{p_j}_{a_j}, L^{q_j}_{b_j})}
\|(\id+\cR_{p_j})^{\frac{a_j-a_z}{\nu_{\cR}}}\|_{\sL(L^{p_j})}.
$$
And since  $\Re (b_z-b_j) \leq 0$
and $\Re (a_j-a_z)\leq 0$, 
Corollary \ref{cor_Riesz_Bessel_opnorm_power_I+R}
then implies
$$
\sup_{y\in \bR} e^{- 3|y|} \ln \|T_{j+iy}\|_{\cL(L^{p_j}, L^{q_j}) } <\infty
,\quad j=0,1
.
$$

The end of the proof is now classical.
We  fix a non-negative function $\chi\in \cS(G)$ with $\int_G \chi=1$
and write
$ \chi_\epsilon(x):= \epsilon^{-Q} \chi(\epsilon^{-1} x) $ for $\epsilon>0$.
If $f\in \sB$, one can show easily  that $f*\chi_\epsilon\in \cS(\bR^n)$ 
and we can define
$T_{z,\epsilon} f :=  T_z\left(f*\chi_\epsilon\right)$ for any $\epsilon>0$, $z\in S$.
Clearly $T_{z,\epsilon}$ satisfy the hypotheses of the Stein-Weiss interpolation theorem,
see \cite[ch. V \S 4]{stein+weiss}.
 Thus for any $t\in [0,1]$, there exists a constant  $M_t>0$ independent of $\epsilon$ such that
 $\|T_{t,\epsilon} f\|_{q_t}\leq M_t \|f\|_{p_t}$
 for any $ f\in \sB$. 

For $p\in (1,\infty)$,
let $\cV_p$ be  the space  of functions $\phi$ 
of the form $\phi=f*\chi_\epsilon$,
with $f\in \sB$ and $\epsilon>0$,
satisfying $\|f\|_p \leq 2\|f*\chi_\epsilon\|_p$.
It is easy to show that
the space $\cV_p$ contains  $\cS(G)$ 
and is dense in $L^p(G)$.
We have obtained 
for any $t\in[0,1]$ and $\phi=f*\chi_\epsilon\in \cV_{p_t}$, that
$$
\|T_t \phi\|_{q_t} = 
 \|T_{t,\epsilon} f\|_{q_t}\leq M_t \|f\|_{p_t}
\leq 2M_t \|\phi\|_{p_t}
.
$$ 
This shows that $T_t$ extends to a bounded operator from $L^{p_t}(G)$ to $L^{q_t}(G)$.
\end{proof}

\subsection{Differential operators acting on Sobolev spaces}
\label{sec_diff_op_on_Lps}

In this section we study how differential operators act on Sobolev spaces.

\begin{theorem}
\label{thm_TR_RT_kerneloftype0}
Let $T$ be any homogeneous left-invariant differential operator of homogeneous degree $\nu_T>0$.
Then for every $p\in (1,\infty)$, 
the operators 
$T \cR_p^{-\frac{\nu_T} \nu}$ 
and
$\cR_p^{-\frac{\nu_T} \nu} T$ 
are of type 0 and, consequently, extend to continuous operators on $L^p(G)$.

Furthermore, $T$ maps continuously $L^p_{s+\nu_T}(G)$ to $L^p_s(G)$ for every $s\in \bR$,
and if $s>0$, 
there exists a constant $C=C_{s,T}>0$ such that 
$$
\forall \phi\in \cS(G)\qquad
\|\cR_p^{\frac {s -\nu_T} \nu} T \phi\|_p \leq C \|\cR^{\frac s\nu } \phi\|_p.
$$
\end{theorem}
\begin{proof}
Let us fix $\alpha\in \bN_0^n \backslash\{0\}$.
Proceeding as in the proof of Corollary \ref{cor_Riesz_Bessel}, 
we can show easily that, for any $a\in \bC$ with  $\Re a-Q-[\alpha]<0$, 
 the integral 
$$
\cI_{a,\alpha} (x):=\frac1{\Gamma(a/\nu)} \int_0^\infty t^{\frac a\nu -1} X^\alpha h_t(x) dt,
$$
converges absolutely for $x\not=0$, 
and in this case it defines a function $\cI_{a,\alpha}$ 
which is smooth away from the origin and $\Re (a-Q-[\alpha])$-homogeneous.
Furthermore, $\cI_{a,\alpha}=X^\alpha \cI_a$ if $\Re a\in (0,Q)$.
Since $\cI_a$ is a distribution, this shows that in this case $\cI_{a,\alpha}$ is also a distribution.
Hence if $\Re a \in (0,Q)$ and $\Re a-Q-[\alpha]<0$ then $\cI_{a,\alpha}$ is a kernel of type $a-[\alpha]$.

Let $N = \lceil \frac{Q+[\alpha]}{\nu}\rceil$. 
We also fix a function $\psi\in \cD(G)$ with $\int_G \psi=1$.
We set $\psi_t(x):= t^{-Q} \psi(t^{-1}x)$.
For any $\phi\in \cR^N(\cS(G))$,
the map $a \mapsto (\cR_p^{-\frac a\nu} \phi) * \tilde X^\alpha \psi_t$
is holomorphic on $\{|\Re a|<N\}$.
On $\{\Re a\in (0,Q)\}$ it coincides with 
$a \mapsto (\phi * \cI_a) * \tilde X^\alpha \psi_t$.
But we see that
$$
(\phi * \cI_a) * \tilde X^\alpha \psi_t
=
\left(X^\alpha (\phi * \cI_a) \right)*  \psi_t
=
(\phi * X^\alpha \cI_a) *  \psi_t
=
(\phi *  \cI_{a,\alpha}) *  \psi_t,
$$
and it is  not difficult to check that 
$a\mapsto (\phi *  \cI_{a,\alpha}) *  \psi_t$
is holomorphic on $\{0<\Re a<Q-\alpha\}$
and continuous on $\{\Re a=Q-\alpha\}$.
Therefore, we have obtained
$$
(\cR_p^{-\frac {[\alpha]}\nu} \phi) * \tilde X^\alpha \psi_t
=
(\phi *  \cI_{[\alpha],\alpha}) *  \psi_t
.
$$
Letting $t\to 0$, we obtain that 
$X^\alpha \cR_p^{-\frac {[\alpha]}\nu}$
coincide with the convolution operator with 
the right-convolution kernel 
$\cI_{[\alpha],\alpha}$,
therefore it is an operator of type 0.
This is so for any $\alpha\in \bN_o^n$.
Consequently 
since any left-invariant $\nu_T$-homogeneous differential operator $T$ on $G$
is  a linear combination of $X^\alpha$ with $[\alpha]=\nu_T$,
 $T \cR_p^{-\frac {[\alpha]}\nu} $ also admits a kernel of type 0.
Necessarily it is also the case for its dual operator 
$\bar \cR_{p'}^{-\frac {[\alpha]}\nu} T^t$. 
This shows the first part of the statement for $\cR_p^{-\frac {[\alpha]}\nu} T$.

Now let us apply this to the operator $\cR^N T$ for $N\in \bN_0$:
the  operator
$\cR^N T \cR_p^{-\frac{\nu_T} \nu -N}$ 
extends to an $L^p(G)$-bounded operator for every $p\in (1,\infty)$.
Since $\cR_p^{N+ \frac {\nu_T}\nu}$ is injective, 
we obtain
$$
\forall \psi \in \cS(G) \qquad
\|\cR^N T \psi \|_{L^p(G)} 
\leq C_N 
\|\cR_p^{\frac{\nu_T} \nu +N} \psi \|_{L^p(G)}.
$$
Consequently
$$
\|T \psi \|_p
+
\|\cR^N T \psi \|_p
\leq 
C_0 
\| \cR_p^{\frac{\nu_T} \nu} \psi \|_p +
C_N 
\|\cR_p^N \cR_p^{\frac{\nu_T} \nu}  \psi \|_p.
$$
By Theorem \ref{thm_sobolev_spaces}, Part (3),
the left-hand side is equivalent to the Sobolev norm of $T\psi$
 in $L^p_{\nu N}(G)$
whereas the following shows that 
 the right-hand side is equivalent to the Sobolev norm of $\psi$
 in $L^p_{\nu_T+\nu N}(G)$.
 Indeed, we have by Theorem \ref{thm_fractional_power}, Part (1), that
$$
\| \cR_p^{\frac{\nu_T} \nu} \psi \|_p 
\leq C
\|  \psi \|_p^{1-\theta}
\| \cR_p^{\frac{\nu_T} \nu +N} \psi \|_p^\theta
\leq C \max 
\left(\|  \psi \|_p,\| \cR_p^{\frac{\nu_T} \nu +N} \psi \|_p\right),
$$
where $\theta=(\frac{\nu_T} \nu)/(\frac{\nu_T} \nu +N)$ since 
$a^{1-\theta}b^\theta\leq \max (a,b)$ for every $a,b\geq 0$ and $\theta \in [0,1]$.

Therefore $T$ is continuous from $L^p_{\nu N + \nu_T}(G)$ to
$L^p_{\nu N }(G)$.
By interpolation (see Theorem \ref{thm_sobolev_spaces_interpolation}), 
it is also continuous 
from $L^p_{s + \nu_T}(G)$ to
$L^p_{s}(G)$
for every $s\geq 0$.
Again by duality, see Corollary \ref{cor_sobolev_spaces_duality},
this shows that this is also true for $s\leq0$.

Since $T$ is continuous from $L^p_{s }(G)$ to
$L^p_{s-\nu_T}(G)$ for $s>\nu_T$, 
there exists $C=C_{s,T}>0$ such that for any $\phi\in \cS(G)$,
$$
\|T\phi\|_p + \|\cR^{\frac {s-\nu_T} \nu} T\phi\|_p \leq C
\left(\|\phi\|_p + \|\cR^{\frac s  \nu} \phi\|_p\right).
$$
In particular applying this to $\phi\circ D_r$ for $r>0$, we obtain after simplification:
$$
r^{\nu_T}\|T\phi\|_p + r^s \|\cR^{\frac {s-\nu_T} \nu} T\phi\|_p 
\leq C
\left(\|\phi\|_p + r^s \|\cR^{\frac s  \nu} \phi\|_p\right).
$$
Since this is true for any $r>0$, by dividing by $r^s$ and letting $r\to\infty$, 
we obtain
$$
 \|\cR^{\frac {s-\nu_T} \nu} T\phi\|_p 
\leq C
 \|\cR^{\frac s  \nu} \phi\|_p.
$$
This concludes the proof of Theorem \ref{thm_TR_RT_kerneloftype0}.
\end{proof}

\subsection{Independence with respect to Rockland operators, and integer orders}

In this Section, 
we show that the Sobolev spaces do not depend on a particular choice of a Rockland operator. 
Consequently Theorems \ref{thm_sobolev_spaces} and \ref{thm_sobolev_spaces_interpolation}, Corollary \ref{cor_sobolev_spaces_cD_dense}, 
and Proposition \ref{prop_construction_fcn_sobolov_spaces}
hold independently of any chosen Rockland operator $\cR$.

We will need the following property:
\begin{lemma}
\label{lem_Lpnuell}
Let $\cR$ be a Rockland operator on $G$ of homogeneous degree $\nu$
and let $\ell\in \bN_0$, $p\in (1,\infty)$.
Then the space
 $L^p_{\nu\ell}(G)$ is the collection of functions 
$f\in L^p(G)$ such that $X^\alpha f\in L^p(G)$ for any $\alpha\in \bN_0^n$ with $[\alpha]=\nu\ell$.
Moreover
the map $\phi\mapsto \|\phi\|_p+\sum_{[\alpha]=\nu \ell} \|X^\alpha \phi\|_p$ is a norm on $L^p_{\nu \ell}(G)$ which is equivalent to the Sobolev norm.
\end{lemma}
\begin{proof}[Proof of Lemma \ref{lem_Lpnuell}]
Writing $\cR^\ell=\sum_{[\alpha]=\ell\nu } c_{\alpha,\ell} X^\alpha$
we have on one hand, 
$$
\exists C>0\quad \forall \phi\in \cS(G)\qquad
\|\cR^\ell \phi\|_p \leq \max |c_\alpha| \sum_{[\alpha]=\ell\nu }  \|X^\alpha \phi\|_p.
$$
Adding $\|\phi\|_{L^p}$ on both sides of this inequality
 implies
by Theorem \ref{thm_sobolev_spaces}, part
\eqref{item_thm_sobolev_spaces_s>0}, that
$$
\exists C>0\quad \forall \phi\in \cS(G)\qquad
\|\phi\|_{L^p_{[\alpha]}}
\leq C\left( \|\phi\|_{L^p}+
\sum_{[\alpha]=\ell\nu } \|X^\alpha \phi\|_p\right).
$$
On the other hand, by Theorem \ref{thm_TR_RT_kerneloftype0},
for any $\alpha\in \bN_0^n$, 
the operator $X^\alpha$ maps continuously $L^p_{[\alpha]}(G)$ to $L^p(G)$,
hence
$$
\exists C>0\quad \forall \phi\in \cS(G)\qquad
\sum_{[\alpha]=\ell\nu } \|X^\alpha \phi\|_p \leq C\|\phi\|_{L^p_{[\alpha]}}.
$$
Lemma \ref{lem_Lpnuell} follows from these estimates.
\end{proof}

One may wonder whether Lemma \ref{lem_Lpnuell}
would be true not only for integer exponents 
of the form $s=\nu\ell$ but for any integer $s$.
In fact other Sobolev spaces on a graded Lie group
were defined by Goodman in \cite[Sec. III. 5.4]{goodman_LNM} 
following this idea. See Section \ref{subsec_comparison}.

We can now show the main result of this section, that is, 
that the Sobolev spaces on graded groups are independent of the chosen positive Rockland operators.

\begin{theorem}
\label{thm_Lps_indep_cR}
For each $p\in (1,\infty)$,
the $L^p$-Sobolev spaces on $G$  
associated with any positive Rockland operators coincide.
Moreover  the Sobolev norms associated to two positive Rockland operators are equivalent.  
\end{theorem}

\begin{proof}[Proof of Theorem \ref{thm_Lps_indep_cR}]
Let $\cR_1$ and $\cR_2$ be two positive Rockland operators on $G$ of homogeneous degree $\nu_1$ and $\nu_2$, respectively.
Then $\cR_1^{\nu_2}$ and $\cR_2^{\nu_1}$ are two positive Rockland operators
 with the same homogeneous degree $\nu=\nu_1\nu_2$.
Their associated Sobolev spaces of exponent $\nu\ell=\nu_1\nu_2\ell$
for any $\ell\in \bN_0$ coincide
and have equivalent norms by Lemma \ref{lem_Lpnuell}.
By interpolation (see Theorem \ref{thm_sobolev_spaces_interpolation}), 
this is true for any Sobolev spaces of exponent $s\geq 0$,
and by duality for any exponent $s\in \bR$. 
\end{proof}

\begin{corollary}
\label{cor_thm_Lps_indep_cR}
Let $\cR^{(1)}$ and $\cR^{(2)}$ be two positive Rockland operators 
on  $G$ with degrees of homogeneity $\nu_1$ and $\nu_2$.
Then for any $s\in \bR$, 
the operator $(\id+\cR^{(1)})^{\frac s {\nu_1}}(\id+\cR^{(2)})^{-\frac s {\nu_2}}$
extends boundedly on $L^p(G)$, $p\in (1,\infty)$.
\end{corollary}

\begin{proof}[Proof of Corollary \ref{cor_thm_Lps_indep_cR}]
We view the operator $(\id+\cR_p^{(2)})^{-\frac a{\nu_2}}$
as a bounded operator from $L^p(G)$ to $L^p_a(G)$
and use the norm $f\mapsto \|(\id+\cR_p^{(1)})^{\frac a {\nu_1}} f\|_p$ on $L^p_a(G)$.
\end{proof}

Thanks to Theorem \ref{thm_Lps_indep_cR},
we can now improve our duality result given in Corollary \ref{cor_sobolev_spaces_duality}:
\begin{proposition}
\label{prop_sobolev_spaces_duality}
 Let $L^p_s(G)$, $p\in [1,\infty)\cup\{\infty_o\}$ and $s\in \bR$,  be the Sobolev spaces on a graded group $G$.
 
For any $s\in \bR$ and $p\in [1,\infty)$,
the dual space of $L^p_{s}(G)$ is isomorphic to $L^{p'}_{-s}(G)$ 
via the distributional duality,
where 
$p'$ is the conjugate exponent of $p$ if $p\in (1,\infty)$, 
and $p'=\infty_o$ if $p=1$.

For any $s\leq 0$ and $p=\infty_o$,
the dual space of $L^{\infty_o}_{s}(G)$ is isomorphic to $L^1_{-s}(G)$ 
via the distributional duality.

If $p\in (1,\infty)$ then the Banach space $L^p_{s}(G)$ is reflexive.
It is also the case for $s\leq 0$ and $p=\infty_o$, and for $s\geq 0$ and $p=1$.
\end{proposition}

We can also show that multiplication by a bump function 
is continuous on Sobolev spaces:
\begin{proposition}
\label{prop_mult_sob_spaces}
For any $\phi\in \cD(G)$, $p\in (1,\infty)$ and $s\in \bR$,
the operator $f \mapsto f\phi$ defined for $f\in \cS(G)$
extends continuously into a bounded map from 
$L^p_s(G)$ to itself.
\end{proposition}

\begin{proof}
The Leibniz' rule for the $X_j$'s and the continuous inclusions in Theorem \ref{thm_sobolev_spaces}
\eqref{item_thm_sobolev_spaces_inclusions}
imply easily that for any fixed $\alpha\in \bN_0^n$ 
there exist a constant $C=C_{\alpha, \phi}>0$ and a constant $C'=C'_{\alpha, \phi}>0$ such that
$$
\forall f\in \cD(G)\quad
\|X^\alpha (f\phi)\|_p\leq C \sum_{[\beta]\leq [\alpha]} \|X^\beta f\|_p
\leq C' \|f\|_{L^p_{[\alpha]}(G)}.
$$
Lemma \ref{lem_Lpnuell} yields the existence of a constant $C"=C"_{\alpha, \phi}>0$
such that
$$
\forall f\in \cD(G)\quad
\|(f\phi)\|_{L^p_{\ell \nu}(G) }\leq C" \|f\|_{L^p_{\ell \nu}(G) }
$$
for any integer $\ell\in \bN_0$
and any  degree of homogeneity  $\nu$  of a Rockland operator.

This shows the statement for the case $s=\nu\ell$.
The case $s>0$ follows by interpolation (see Theorem \ref{thm_sobolev_spaces_interpolation}),
and the case $s<0$ by duality (see Proposition \ref{prop_sobolev_spaces_duality}).
\end{proof}

\subsection{Properties of $L^2_s(G)$}

The case $L^2(G)$ has some special features, 
mainly being a Hilbert space,
that we will discuss here.

\medskip

Many of the proofs in this paper
could be simplified if we had just considered the case $L^p$ with $p=2$.
For instance, 
let us consider a positive Rockland operator $\cR$
and its self-adjoint extension $\cR_2$ on $L^2(G)$.
One can define 
the fractional powers of $\cR_2$ and $\id+\cR_2$
by functional analysis.
Then one can obtain the properties of the kernels of the Riesz and Bessel potentials with similar methods as in Corollary \ref{cor_Riesz_Bessel}.

The proof of the properties of the associated Sobolev spaces $L^2_s(G)$ would be the same in this particular case, maybe slightly helped occasionally 
by the H\"older inequality being replaced by the Cauchy-Schwartz inequality.
A noticeable exception is that Lemma \ref{lem_Lpnuell}
can be obtained directly in the case $L^p$, $p=2$,
from the estimates due to Helffer and Nourrigat in \cite{helffer+nourrigat-79}.

The main difference between $L^2$ and $L^p$ Sobolev spaces
is the structure of Hilbert spaces of  $L^2_s(G)$
whereas the other Sobolev spaces $L^p_s(G)$ are `only' Banach spaces:

\begin{proposition}[Hilbert space $L^2_s$]
\label{prop_sobolev_spaces_p=2}
Let $G$ be a graded group.

For any $s\in \bR$,
$L^2_s(G)$ is a Hilbert space with inner product given by
$$
(f,g)_{L^2_s(G)} := 
\int_G  (\id+\cR_2)^{\frac s\nu} f (x)\ \overline{ (\id+\cR_2)^{\frac s\nu} g(x)} dx,
$$ 
where $\cR$ is a positive Rockland operator of homogeneous degree $\nu$.
 
If $s>0$, an equivalent inner product is
$$
(f,g)_{L^2_s(G)} := 
\int_G  
f (x)\ \overline{ g(x)} dx
\ + \
\int_G  \cR_2^{\frac s\nu} f (x)\ \overline{ \cR_2^{\frac s\nu} g(x)} dx
.
$$ 

If $s=\nu\ell$ with $\ell\in \bN_0$, an equivalent inner product is
  $$
  (f,g)= (f,g)_{L^2(G)}+
  \sum_{[\alpha]=\nu\ell} (X^\alpha f,X^\alpha g)_{L^2(G)}
  .
  $$
\end{proposition}

Proposition \ref{prop_sobolev_spaces_p=2}
is easily checked, using the structure of Hilbert space of $L^2(G)$.

 \section{Further properties of Sobolev spaces}
 
In this section we show a Sobolev embedding theorem, 
and this will require showing that
the operators of type 0 act continuously on Sobolev spaces. 
We also compare the spaces we have constructed in the previous section
with other possible definitions of Sobolev spaces.

\subsection{Operators of type 0 acting on Sobolev spaces}
 
In this section we show that the result given in Theorem \ref{thm_op_type_Lp_bdd} for operator of type 0 can be extended to Sobolev spaces:

\begin{theorem}
\label{thm_conv_type0_sobolev}
Any operator of type $\nu_o$ with $\Re \nu_o=0$,
 extends to a bounded operator on $L^p_s(G)$ for any $p\in (1,\infty)$ and $s\in \bR$.
\end{theorem}

In the statement and in the proof, 
we  keep the same notation for an operator on $\cD(G)\to\cD'(G)$ and its possible bounded extensions to some Sobolev spaces
in order to ease the notation.

Before giving the proof of Theorem \ref{thm_conv_type0_sobolev}, 
let us comment on  similar results in related contexts.
In the case of $\bR^n$ (and similarly for compact Lie  groups), 
the continuity on Sobolev spaces would be easy since $T_\kappa$ would commute with the Laplace operator
but the homogeneous setting requires a more substantial argument.
On any stratified group,
Theorem \ref{thm_conv_type0_sobolev} was shown by Folland
in \cite[Theorem 4.9]{folland_75}.
However the proof in this context uses 
the existence of a positive Rockland operator
with a unique homogeneous fundamental solution, 
namely `the' (any) sub-Laplacian.
If we wanted to follow closely the same line of arguments, 
we would have to assume that the group is equipped 
with a Rockland operator of homogeneous degree $\nu$ with $\nu<Q$.
This is not always the case for a graded group
(it suffices to consider for example  the three dimensional Heisenberg group $\tilde \bH_1$ with a graded non-stratified structure defined in 
Section \ref{subsec_comparison}).
We present here a proof which is valid 
under no restriction in the graded case. 
As in the stratified case, the main problem is to check at every step that formal convolutions between different kernels make sense,
see the discussion before 
Proposition \ref{prop_conv_kernel_typnu}.

\begin{proof}[Proof of Theorem \ref{thm_conv_type0_sobolev}]
Let $\kappa$ be a kernel of type $\nu_o$ with $\Re \nu_o=0$ 
 and let $\cR$ be a positive Rockland operator of 
homogeneous degree $\nu$.
By Corollary \ref{cor_Riesz_Bessel} (i),
if $a\in (0,Q)$,
$\cI_a$ is a kernel of type $a$,
and for any $\phi\in \cS(G)$,
we have
$$
\cR_p^{\frac a\nu} \phi\in L^p(G)\cap L^{\tilde q}(G)
\quad\mbox{and}\quad\phi=(\cR_{\tilde q}^{\frac a\nu} \phi)*\cI_a
=
(\cR_p^{\frac a\nu} \phi)*\cI_a,
$$
where, for instance,  $\tilde q =\frac 12 (1+ \frac Q a) < \frac Q a$.
By Proposition~\ref{prop_conv_kernel_typnu} (ii),
 $\cI_{a}*\kappa$ is a kernel of type $a+\nu_o$
 with $\Re (a+\nu_o)=a\in (0,Q)$,
and 
$$
T_\kappa \phi=\phi *\kappa = 
\left((\cR_p^{\frac a\nu} \phi)*\cI_a\right)*\kappa
=
(\cR_p^{\frac a\nu} \phi)*\left(\cI_a*\kappa\right)
\quad
(\mbox{in some}\ L^q(G)).
$$
This implies
that  for any $j=1,\ldots,n$, 
$\kappa_j:= X_j \left(\cI_{\upsilon_j}*\kappa\right)$
is a kernel of  type $\nu_o$ and that
the following operators coincide on $\cS(G)$
$$
X_j T_\kappa= T_{\kappa_j}\cR^{\frac {\upsilon_j}\nu}.
$$
Since $T_{\kappa_j}$ is $L^p$-bounded  
(see Theorem~\ref{thm_op_type_Lp_bdd}),
we have obtained for any $j=1,\ldots,n$ and any $\phi\in \cS(G)$:
$$
\|X_j T_\kappa \phi\|_p=
\|T_{\kappa_j}\cR^{\frac {\upsilon_j}\nu}\|_p
 \leq C \|\cR^{\frac {\upsilon_j}\nu}\phi\|_p
\leq C' \|\phi\|_{L^p_{\upsilon_j}},
$$
using Theorem \ref{thm_sobolev_spaces}
\eqref{item_thm_sobolev_spaces_s>0}
for the last inequality.
Note that this yields for any two indices $j_1,j_2=1,\ldots,n$,
$$
\|X_{j_2}X_{j_1} T_\kappa \phi\|_p 
\leq 
C_1 \|X_{j_1}\phi\|_{L^p_{\upsilon_{j_2}}}
\leq 
C_2 \|\phi\|_{L^p_{\upsilon_{j_2}+\upsilon_{j_1}}}
,
$$
since $X_{j_1}$ maps $L^p_{s+\upsilon_{j_1}}$ to $L^p_s$ boundedly
(see Theorem \ref{thm_TR_RT_kerneloftype0}).
Recursively, 
writing any $X^\alpha$
as the composition of various $X_j$ yields
$$
\exists C=C_\alpha\quad\forall \phi\in \cS(G)\qquad
\|X^\alpha T_\kappa \phi\|_p 
\leq 
C \|\phi\|_{L^p_{[\alpha]}}
.
$$
This is true for any $\alpha\in \bN_0^n$,
the case $\alpha=0$ following from 
Theorem~\ref{thm_op_type_Lp_bdd}.
For each $\ell\in \bN_0$ fixed, 
we now sum  over $[\alpha]=0,\ell \nu$, to get
$$
\|T_\kappa \phi\|_p +\sum_{[\alpha]= \ell \nu}
\|X^\alpha T_\kappa \phi\|_p 
\leq 
C \left(\|\phi\|_{L^p_{0}} + \sum_{[\alpha]=\ell \nu}
\|\phi\|_{L^p_{[\alpha]}}\right)
\leq 
C' 
\|\phi\|_{L^p_{\ell\nu}}.
$$
The left hand side is equivalent to $\|T_\kappa\phi\|_{L^p_{\nu\ell}}$ by  Lemma \ref{lem_Lpnuell}.
Thus we obtain
$$
\exists C=C_\ell\quad\forall \phi\in \cS(G)\qquad
\| T_\kappa \phi\|_{L^p_{\ell\nu}}
\leq 
C \|\phi\|_{L^p_{\ell\nu}}.
$$

We have obtained that,
for any kernel $\kappa$ of type 0,
the corresponding convolution operator
  $T_\kappa$ maps continuously $L^p_s$ to itself
for any $s=\ell\nu$ with $\ell\in \bN_0$ and $p\in (1,\infty)$.
The result for any $s\in \bR$ follows 
by interpolation (see Theorem~\ref{thm_sobolev_spaces_interpolation}),
and duality (see Proposition \ref{prop_sobolev_spaces_duality}).
This concludes the proof of Theorem~\ref{thm_conv_type0_sobolev}.
\end{proof}

\subsection{Sobolev embedding theorem}

In this section, we show
the analogue of the classical fractional integration theorems of Hardy-Littlewood and Sobolev.
The main difference is that the topological dimension $n$ of $G\sim \bR^n$
is replaced by  the homogeneous dimension $Q$.
The stratified case was proved 
by Folland
in \cite{folland_75} (mainly Theorem 4.17 therein).

\begin{theorem}
\label{thm_Sobolev-ineq}
\begin{itemize}
\item[(i)] If $1<p<q<\infty$ and $a,b\in \bR$
with $b-a=Q(\frac 1p-\frac 1q)$ then
we have the following continuous inclusion:
$$
L^p_b\subset L^q_a,
$$
that is, for every $f\in L^p_b$,
we have $f\in L^q_a$
and there exists a constant $C=C_{a,b,p,q,G}>0$ independent of $f$ such that 
$$
\|f\|_{L^q_a} \leq C \|f\|_{L^p_b}.
$$
\item[(ii)] If $p\in (1,\infty)$ and $s>Q / p$
then we have the following inclusion:
$$
L^p_s \subset \left(C(G)\cap L^\infty(G)\right),
$$
in the sense that any function $f\in L^p_s(G)$ admits a bounded continuous representative (still denoted by $f$). Furthermore  
there exists a constant $C=C_{s,p,G}>0$ independent of 
$f$ such that 
$$
\|f\|_\infty \leq C \|f\|_{L^p_s(G)}.
$$
\end{itemize}
\end{theorem}

\begin{proof}
Let us prove Part (i). 
Note that, under the condition of the statement,  
$b-a\in (0,Q)$ and that the relation between $p$ and $q$ is 
the one giving the $L^p\to L^q$-continuity of operator of type $b-a$
by Theorem \ref{thm_op_type_Lp_bdd}.

We fix a positive Rockland operator $\cR$ of homogeneous degree $\nu$
and we assume that $b,a>0$ and $p,q\in (1,\infty)$
satisfy $b-a=Q(\frac 1p-\frac 1q)$.

By Corollary \ref{cor_Riesz_Bessel} (i), 
$\cI_{b-a}$ is a kernel of type $b-a$
and for any $p_1\in (1,\infty)$ and  $\phi\in \cS(G)$, 
$\cR_{p_1} ^{\frac {b-a}\nu} \phi \in L^{p_1}_b$ and 
$\phi=(\cR_{p_1}^{\frac {b-a}\nu} \phi) *\cI_{b-a}$.
By Theorem \ref{thm_op_type_Lp_bdd},
this implies with $p_1=p$,
$$
\|\phi\|_q \leq C\|\cR_{p}^{\frac {b-a}\nu} \phi\|_{L^p}.
$$

For the same reason we also have $\cR_q^{\frac a\nu} \phi=\cR_p^{\frac b\nu} *\cI_{b-a}$ and 
$$
\|\cR_q^{\frac a\nu}\phi\|_q \leq C\|\cR_{p}^{\frac {b}\nu} \phi\|_{L^p}.
$$

Adding the two estimates above, we obtain
$$
\|\phi\|_q+\|\cR_q^{\frac a\nu}\phi\|_q 
\leq C
\left(\|\cR_{p}^{\frac {b-a}\nu} \phi\|_{L^p}+
\|\cR_{p}^{\frac {b}\nu} \phi\|_{L^p}\right).
$$
Since $b$, $a$, and  $b-a$ are positive, 
by  Theorem \ref{thm_sobolev_spaces}
\eqref{item_thm_sobolev_spaces_inclusions},
the left-hand side is equivalent to 
$\|\phi\|_{L^q_a}$
and both terms in the right-hand side are $\leq C\|\phi\|_{L^q_b}$.
Therefore we have obtained:
$$
\exists C=C_{a,b,p,q,\cR} \quad\forall \phi\in\cS(G)\qquad
\|\phi\|_{L^q_a}
\leq C
\|\phi\|_{L^q_b}.
$$
By density of $\cS(G)$ in the Sobolev spaces, 
this shows Part (i) for $b>a>0$.
The result  for any $a,b$ follows
by duality and interpolation (see Proposition \ref{prop_sobolev_spaces_duality} and Theorem \ref{thm_sobolev_spaces_interpolation}).
The proof of Part (i) is now complete.

Let us prove Part (ii).
Let $p\in (1,\infty)$ and $s>Q/p$.
By Corollary \ref{cor_Riesz_Bessel} (ii),
we know $\cB_s \in L^1(G) \cap L^{p'}(G)$,
where $p'$ is the conjugate exponent of $p$.
For any $f\in L^p_s(G)$,
we have $f_s:=(\id+\cR_p)^{\frac s\nu} f\in L^p$
and
$$
f=
(\id+\cR_p)^{-\frac s\nu} f_s
=
f_s*\cB_s.
$$
Therefore by H\"older's inequality,
$$
\|f\|_\infty \leq \|f_s\|_p \|\cB_s\|_{p'}
=\|\cB_s\|_{p'} \|f\|_{L^p_s}.
$$
Moreover for almost every $x$, we have
$$
f(x)=\int_G f_s(y) \cB_s(y^{-1}x)dy
=\int_G f_s(x z^{-1}) \cB_s(z) dz
.
$$
Thus for almost every $x,x'$, we have
\begin{eqnarray*}
|f(x)-f(x')|
&=&
\left|\int_G \left( f_s(x z^{-1}) -  f_s(x' z^{-1}) \right)\cB_s(z) dz \right|
\\
&\leq&
\|\cB_s\|_{p'} \|f_s(x\,\cdot)-f_s(x'\,\cdot)\|_{p}.
\end{eqnarray*}
Translation is continuous on $L^p(G)$, 
thus as $x'\to x$,
$ \lim \|f_s(x\,\cdot)-f_s(x'\,\cdot)\|_{L^p(G)}=0$
and consequently $|f(x)-f(x')|\longrightarrow 0$
almost surely.
Hence we can modify $f$ so that it becomes a continuous function.
This concludes the proof.
\end{proof}

From the Sobolev embedding theorem (Theorem \ref{thm_Sobolev-ineq} (ii)) and the description of Sobolev spaces with integer exponent (Lemma \ref{lem_Lpnuell}) follows easily the following property:
\begin{corollary}
\label{cor_thm_Sobolev_ineq}
Let $G$ be a graded group, $p\in (1,\infty)$ and $s\in \bN$.
We assume that $s$ is proportional to the homogeneous degree $\nu$ of  a positive Rockland operator, that is,  $\frac s\nu \in \bN$, and that $s>Q/p$.

Then if $f$ is a distribution on $G$ such that $f\in L^p(G)$
and $X^\alpha f\in L^p(G)$ when  $\alpha\in \bN_0^n$ satisfies $[\alpha]=s$, 
then $f$ admits a bounded continuous representative (still denoted by $f$). Furthermore  
there exists a constant $C=C_{s,p,G}>0$ independent of 
$f$ such that 
$$
\|f\|_\infty \leq C \left(\|f\|_p+\sum_{[\alpha]=s} \|X^\alpha f\|_p\right).
$$
\end{corollary}

\subsection{Comparison with other definitions of Sobolev spaces}
\label{subsec_comparison}

If the group $G$ is stratified, 
then we can choose as positive Rockland  operator $\cR=-\cL$ 
with $\cL$ a (negative) sub-Laplacian.
The corresponding Sobolev spaces have been developed by Folland in  \cite{folland_75} for stratified groups, see also \cite{saka}.
Folland showed that his Sobolev spaces do not depend on a particular choice of a sub-Laplacian
\cite[Corollary 4.14]{folland_75},
and we have shown the same for our Sobolev spaces and Rockland operators
in Theorem \ref{thm_Lps_indep_cR}.
Therefore, our Sobolev spaces coincide with Folland's in the stratified case,
and gives new descriptions of Folland's Sobolev spaces.

For instance, let us consider the `simplest' case after the abelian case, 
that is, the three dimensional Heisenberg group $\bH_1$,
with Lie algebra $\fh_1=\bR X\oplus \bR Y\oplus\bR T$
and canonical commutation relations  $[X,Y]=T$.
This is naturally a stratified group, with canonical (negative) sub-Laplacian $\cL_{\bH_1}:=X^2+Y^2$.
We have obtained that the Sobolev spaces 
(in our sense or equivalently Folland's) 
may be defined using 
any of the positive Rockland operators 
$$
-\cL_{\bH_1},\quad\mbox{or}\ \cL_{\bH_1}^2
,\quad\mbox{or}\ \cL_{\bH_1}^2-T^2.
$$

To compare our Sobolev spaces $L^p_s(G)$ with their Euclidean counterparts $L^p_s(\bR^n)$,
that is, for the abelian group $(\bR^n,+)$,
we can proceed as in  \cite{folland_75}, especially Theorem 4.16 therein.
First there can be only local relations between our Sobolev Spaces and the Euclidean Sobolev spaces, since the coefficients of $X_j$'s with respect to the abelian derivatives $\partial_{x_k}$
are polynomials in the coordinate functions $x_\ell$'s,
and conversely, the coefficients of $\partial_{x_j}$'s with respect to the abelian derivatives $X_k$
are polynomials in the coordinate functions $x_\ell$'s.
Hence we are led to define the following local Sobolev spaces for $s\in \bR$ and $p\in (1,\infty)$:
\begin{equation}
\label{eq_def_Lpsloc}
L^p_{s,loc}(G):=\{f\in \cD'(G) : \phi f \in L^p_s(G) \ \mbox{for all} \ \phi\in \cD(G)\}.
\end{equation}
By Proposition \ref{prop_mult_sob_spaces}, 
$L^p_{s,loc}(G)$ contains $L^p_{s}(G)$.
We can compare locally the Sobolev spaces on graded groups and on their abelian counterpart:
\begin{theorem}
\label{thm_inclusion_local_sobolev_spaces}
For any $p\in (1,\infty)$ and $s\in \bR$,
$$
L^p_{s/\upsilon_1 ,loc}(\bR^n) \subset 
L^p_{s,loc}(G) \subset 
L^p_{s/\upsilon_n,loc}(\bR^n).
$$
\end{theorem}

Above, $L^p_{s,loc}(\bR^n)$ denotes the usual local Sobolev spaces,
or equivalently the spaces defined by \eqref{eq_def_Lpsloc} in the case of
the abelian (graded) group $(\bR^n,+)$.
Recall that $\upsilon_1$ and $\upsilon_n$ 
are respectively the smallest and the largest weights of the dilations.
In particular, in the stratified case, $\upsilon_1=1$ 
and $\upsilon_n$  coincides 
with  the number of steps in the stratification,
and with the step of the nilpotent Lie group $G$.
Hence in the stratified case we recover  Theorem 4.16 in \cite{folland_75}.

\begin{proof}[Proof of Theorem \ref{thm_inclusion_local_sobolev_spaces}]
It suffices to show that the mapping $f\mapsto f\phi$ defined on $\cD(G)$
extends boundedly from $L^p_{s/\upsilon_1}(\bR^n)$ to $L^p_s(G)$
and from $L^p_s(G)$ to $L^p_{s/\upsilon_n,loc}(\bR^n)$.
By duality and interpolation
 (see Theorem \ref{thm_sobolev_spaces_interpolation} and Proposition \ref{prop_sobolev_spaces_duality}),
 it suffices to show this for a sequence of increasing positive integers $s$.

For the $L^p_{s/\upsilon_1}(\bR^n) \to L^p_s(G)$ case, 
we assume that $s$ is 
 divisible by the homogeneous degree of a positive Rockland operator.
 Then we use Lemma \ref{lem_Lpnuell}, the fact that 
the $X^\alpha$ may be written as a  combination of the $\partial_{x}^\beta$
 with polynomial coefficients in the $x_\ell$'s
 and that $\max_{[\beta] \leq s} |\beta| = s / \upsilon_1$.
 
 For the case of $L^p_s(G) \to L^p_{s/\upsilon_n,loc}(\bR^n)$,
 we use the fact that the abelian derivative $\partial_x^\alpha$, $|\alpha|\leq s$, 
 may be written as a combination over the $X^\beta$, $|\beta|\leq s$,
 with polynomial coefficients in the $x_\ell$'s, that $X^\beta$ maps $L^p\to L^p_{[\beta]}$
 boundedly together with 
 $\max_{|\beta|\leq s} [\beta] =s\upsilon_n.$
 \end{proof}
 
Proceeding as in  \cite[p.192]{folland_75}, 
one can convince oneself that Theorem \ref{thm_inclusion_local_sobolev_spaces}
can not be improved.

\medskip

In another direction, 
Sobolev spaces, and more generally Besov spaces,
have been defined on any group of polynomial growth
in  \cite{furioli+melzi+veneruso} using left-invariant sub-Laplacians
and an associated Littlewood-Payley decomposition.
Considering stratified groups and homogeneous left-invariant sub-Laplacians 
(as in \eqref{eq_def_lih_subLapl}), 
this gives another description of the Sobolev spaces in the stratified case
which is equivalent to Folland's and to ours.
However, for a general graded non-stratified group, our Sobolev spaces may differ from the ones in \cite{furioli+melzi+veneruso} on any Lie group of polynomial growth.
For instance, if we consider the three dimensional Heisenberg group endowed 
with the dilations
\begin{equation}
\label{eq_dilation358_H1}
r\cdot (x,y,t) = (r^3 x, r^5 y , r^8 t).
\end{equation}
We denote this group  $\tilde \bH_1$, 
it is graded but not stratified.
The sub-Laplacian $\cL_{\bH_1}$ is not homogeneous and is of degree 10.
One can check that  that $\cL_{\bH_1}$ maps $L^2_{10}(\tilde \bH_1)\to L^2(\tilde \bH_1)$
and $L^2_2( \bH_1)\to L^2(\bH_1)$
boundedly and this can not be improved. 
Hence our Sobolev spaces on $\tilde \bH_1$
differ from the Sobolev spaces based on the sub-Laplacian 
in \cite{folland_75} or equivalently
in \cite{furioli+melzi+veneruso}. 

\medskip

Sobolev spaces of integer exponents on graded Lie groups have already been defined 
by Goodman in \cite[Sec. III. 5.4]{goodman_LNM}:
the $L^p$ Goodman-Sobolev spaces  of order 
 $s\in \bN_0$ is the space of function $\phi\in L^p$ such that 
 $X^\alpha \phi\in L^p$ for any $[\alpha]\leq s$.
Goodman's definition does not use Rockland operators but makes sense only   for integer exponents. 
Adapting  the proof of 
Lemma \ref{lem_Lpnuell}, 
one could show easily that the 
$L^p$ Goodman-Sobolev space of order 
 $s\in \bN_0$ always contains our Sobolev space $L^p_s(G)$,
 and in fact coincides with it if  $s$ is proportional to the homogeneous degree $\nu$
 of a positive Rockland operator
 or for any $s$ if  the group is stratified.

However, this equality between Goodman-Sobolev spaces and our Sobolev spaces is not true on any general graded group.
For instance this does not hold on graded Lie groups whose weights are all strictly greater than 1.
Indeed,
 the $L^p$ Goodman-Sobolev space of order 
 $s=1$ is $L^p(G)$ which  contains $L^p_1(G)$  stricly
 (see Theorem \ref{thm_sobolev_spaces}
\eqref{item_thm_sobolev_spaces_inclusions}).
An example of such a graded group is 
the three dimensional Heisenberg group $\tilde \bH_1$ 
 with weights given by \eqref{eq_dilation358_H1}.

One consequence of these strict inclusions together with our results
is that the Goodman-Sobolev spaces do not satisfy interpolation properties in general. 
This together with the fact that, to the authors knowledge, no Sobolev embeddings have been proved
for those spaces, limits the use of the Goodman-Sobolev spaces.

\medskip

Another advantage of the analysis developed in this paper 
is that it is easy to define homogeneous Sobolev spaces $\dot L^p_s(G)$, 
$s\in \bR$ and $p\in (1,\infty)$,
as the completion of $f\mapsto \|\cR_p^{\frac s \nu} f\|_p$
for a Rockland operator $\cR$ of degree $\nu$.
Moreover simple adaptions of the proofs presented here imply
that these spaces satisfy similar properties of inclusions,
interpolation and duality as their Euclidean counterparts.
As in the non-homogeneous case, one also obtains 
that these spaces do not depend on a special choice of Rockland operators $\cR$.

\section*{Acknowledgements}
The first author acknowledges the support of the London Mathematical Society via the Grace Chisholm Fellowship held at King's College London in 2011 
and of the University of Padua.
The second author was supported in part by the EPSRC Leadership Fellowship EP/G007233/1
and by EPSRC Grant EP/K039407/1.


\begin{thebibliography}{00}

	
\bibitem{coifman+weiss-LNM71}
{Coifman, R. and Weiss, G.,}
{\em Analyse harmonique non-commutative sur certains espaces
              homog\`enes},
 {Lecture Notes in Mathematics \textbf{242}},
 {Springer-Verlag},
{Berlin},
1971.

\bibitem{TElst+Robinson}
 {ter Elst, A. F. M. and Robinson, Derek W.},
 {Spectral estimates for positive {R}ockland operators},
 in {Algebraic groups and {L}ie groups},
  {Austral. Math. Soc. Lect. Ser.},
  \textbf {9},
    {195--213},
 {Cambridge Univ. Press},
 {(1997)}.

%


\bibitem{folland_75}
 {Folland, G. B.},
 {Subelliptic estimates and function spaces on nilpotent {L}ie
              groups},
 {\em Ark. Mat.},
 \textbf{13}  {(1975)},
 {161--207}.
		


\bibitem{folland+stein_bk82}
{Folland, G. B. and Stein, E.},
{\em Hardy spaces on homogeneous groups},
{Mathematical Notes \textbf{28}},
{Princeton University Press},
1982.

\bibitem{furioli+melzi+veneruso}
 {Furioli, G. and Melzi, C. and Veneruso, A.},
 {Littlewood-{P}aley decompositions and {B}esov spaces on {L}ie
              groups of polynomial growth},
 {Math. Nachr.},
\textbf{279},
     {2006},
    No {9-10},
    {1028--1040}.
		

\bibitem{geller}
{Geller, D.},
  {Fourier analysis on the {H}eisenberg group. {I}. {S}chwartz
              space},
 {\em J. Funct. Anal.},
\textbf {36},
  {(1980)},
 {205--254}.
		
\bibitem{geller_83}
 {Geller, D.},
 {Liouville's theorem for homogeneous groups},
 {\em Comm. Partial Differential Equations},
    \textbf{8},
   {(1983)},
    No {15},
   {1665--1677}.

\bibitem{glowacki_89}
 {G{\l}owacki, P.},
 {The {R}ockland condition for nondifferential convolution
              operators},
  {Duke Math. J.},
\textbf{58},
       {(1989)},
    No {2},
      {371--395}.

\bibitem{glowacki_91}
 {G{\l}owacki, P.},
  {The {R}ockland condition for nondifferential convolution
              operators. {II}},
  {Studia Math.},
\textbf{98},
     {(1991)},
    No {2},
  {99--114}.


\bibitem{goodman_LNM}
 {Goodman, R.},
 {\it Nilpotent {L}ie groups: structure and applications to
              analysis},
  {Lecture Notes in Mathematics \textbf{562}},
 {Springer-Verlag},
 {1976}.
	

\bibitem{helffer+nourrigat-79}
{Helffer, B. and Nourrigat, J.},
{Caracterisation des op\'erateurs hypoelliptiques homog\`enes
              invariants \`a gauche sur un groupe de {L}ie nilpotent
              gradu\'e},
{\em Comm. Partial Differential Equations},
\textbf{4}  {(1979)},
 {899--958}.
 	
   
\bibitem{hula-1984}
 {Hulanicki, A.},
  {A functional calculus for {R}ockland operators on nilpotent
              {L}ie groups},
   {\em Studia Math.}, 
   \textbf{78} (1984),  253--266.
 
 \bibitem{hunt}
 {Hunt, G. A.},
 {Semi-groups of measures on {L}ie groups},
  {Trans. Amer. Math. Soc.},
\textbf {81},
      {(1956)},
     {264--293}.

\bibitem{martinez+sanz_bk} 
  {Mart{\'{\i}}nez C. and Sanz A.},
 {\it The theory of fractional powers of operators},
     {North-Holland Mathematics Studies},
    \textbf {187},
 {North-Holland Publishing Co.},
 {Amsterdam},
 {2001}.
		
\bibitem{mazya} 
Maz'ya, V., 
{\it Sobolev spaces with applications to elliptic partial differential equations},
Second edition, Grundlehren der Mathematischen Wissenschaften, 
\textbf{342}, Springer, Heidelberg, 2011.
 
\bibitem{miller}
 {Miller, K.},
 {Parametrices for hypoelliptic operators on step two nilpotent
              {L}ie groups},
 {\em Comm. Partial Differential Equations},
 \textbf{5},
  {(1980)},
   No {11},
 {1153--1184}.

\bibitem{rockland}
 {Rockland, C.},
 {Hypoellipticity on the {H}eisenberg
              group-representation-theoretic criteria},
    {Trans. Amer. Math. Soc.},
\textbf{240},
       {(1978)},
      {1--52}.
      
 \bibitem{rt:book}
Ruzhansky, M. and Turunen, V.,
{\em Pseudo-differential operators and symmetries. Background
  analysis and advanced topics}, 
{Pseudo-Differential
  Operators. Theory and Applications},
    \textbf{2}, 
 Birkh{\"a}user Verlag, Basel, 2010.
 
 \bibitem{rt:groups}
Ruzhansky, M. and Turunen, V.,
Global quantization of pseudo-differential operators on compact {L}ie
  groups, {SU}(2), 3-sphere, and homogeneous spaces,
Int. Math. Res. Not. IMRN, 
\textbf{11}, (2013), 2439--2496.
     
      
 \bibitem{saka}
  {Saka, K.},
  {Besov spaces and {S}obolev spaces on a nilpotent {L}ie group},
   {T\^ohoku Math. J. (2)},
\textbf{31}, No 4,
 {1979}.
      
	
\bibitem{stein+weiss}	
 {Stein, E. and Weiss, G.},
 {\it Introduction to {F}ourier analysis on {E}uclidean spaces},
 {Princeton Mathematical Series, No. 32},
 {1971}.


\end{thebibliography}
\end{document}